\documentclass[12pt]{amsart}
\usepackage{graphicx} 
\usepackage{amstext,amsfonts,amsthm,amssymb,amscd,amsbsy,amsmath,verbatim, mathrsfs, fullpage,hyperref,xcolor}

\newtheorem{lemma}{Lemma}

\newcommand{\shrinkmargins}[1]{
  \addtolength{\textheight}{#1\topmargin}
  \addtolength{\textheight}{#1\topmargin}
  \addtolength{\textwidth}{#1\oddsidemargin}
  \addtolength{\textwidth}{#1\evensidemargin}
  \addtolength{\topmargin}{-#1\topmargin}
  \addtolength{\oddsidemargin}{-#1\oddsidemargin}
  \addtolength{\evensidemargin}{-#1\evensidemargin}
  }

\shrinkmargins{.7}

\DeclareMathOperator{\GL}{GL}

\DeclareMathOperator{\Rea}{Re}

\DeclareMathOperator{\Gal}{Gal}

\DeclareMathOperator{\vol}{vol}
\DeclareMathOperator{\covol}{covol}

\newcommand{\field}[1]{\mathbb{#1}}

\newcommand{\Q}{\field{Q}}

\newcommand{\Z}{\field{Z}}

\newcommand{\F}{\field{F}}

\newcommand{\R}{\field{R}}
\newcommand{\C}{\field{C}}
\renewcommand{\P}{\field{P}}

\newcommand{\EE} {\mathbb{E}}
\newcommand{\SSS} {\mathcal{S}}

\newcommand{\ra}{\rightarrow}

\newcommand{\OO}{\mathcal{O}}

\newcommand{\ic}[1]{\mathfrak{#1}}

\newcommand{\LL}{\mathcal{L}}
\newcommand{\FF}{\mathcal{F}}
\newcommand{\KK}{\mathcal{K}}

\newcommand{\tensor} {\otimes}

\newcommand{\bs}{\backslash}

\newcommand{\set}[1]{\{#1\}}

\newcommand{\beq}{\begin{displaymath}}
\newcommand{\eeq}{\end{displaymath}}
\newcommand{\beqn}{\begin{equation}}
\newcommand{\eeqn}{\end{equation}}

\newcommand{\eqd}{\stackrel{d}{=}}

\theoremstyle{plain}
\newtheorem{thm}{Theorem}
\newtheorem{prop}[thm]{Proposition}
\newtheorem{cor}[thm]{Corollary}
\newtheorem{lem}[thm]{Lemma}

\theoremstyle{definition}
\newtheorem{defn}[thm]{Definition}

\newtheorem{exmp}[thm]{Example}
\newtheorem*{question}{Question}

\theoremstyle{remark}
\newtheorem{rem}[thm]{Remark}

\usepackage[numbers]{natbib}

\title{Smyth's conjecture and a non-deterministic Hasse principle}
\author{Jordan S. Ellenberg and Will Hardt}
\date{March 2025}

\begin{document}

\begin{abstract} In a 1986 paper, Smyth proposed a conjecture about which integer-linear relations were possible among Galois-conjugate algebraic numbers.  We prove this conjecture.  The main tools (as Smyth already anticipated) are combinatorial rather than number-theoretic in nature. What's more, we reinterpret Smyth's conjecture as a local-to-global principle for a ``non-deterministic system of equations" where variables are interpreted as compactly supported $K$-valued random variables rather than as elements of $K$.
\end{abstract}

\maketitle

\section{Introduction}

Are there Galois-conjugate algebraic numbers $\theta_1, \theta_2, \theta_3$ such that $17 \theta_1 + 19 \theta_2 + 29 \theta_3 = 0$?  Or such that $\theta_1^{17} \theta_2^{19} \theta_3^{29} = 1$?  In \cite{Smy86}, Smyth proposed a conjecture concerning the possible integer-linear relations between Galois conjugates, both additive and multiplicative.  (In this paper, we consider only additive relations.)  Smyth showed that the existence of such a relation among Galois conjugates was equivalent to the existence of a certain multiset of integer solutions to the relation satisfying a certain combinatorial condition.  Smyth's criterion is easiest to understand by means of an example.  Consider the matrix 
\[
\begin{bmatrix}
3 & -3 &  4&  -4& 5 & -5 & 0 & 0 \\
4 &  -4 & -3 & 3 & 0 & 0 & 5 & -5 \\
-5&  5 &  0&  0& -3 & 3 &  -4& 4 
\end{bmatrix}
\]
which we denote by $A$.  Each column of $A$ is a solution to the relation $3x+4y+5z = 0$.  The rows of $A$ are all permutations of each other.  Smyth showed that the question of which relations could occur between Galois conjugates is exactly the question of which linear relations admit matrices like the one above.  (We will state this precisely in Proposition~\ref{pr:smythcriterion}.)  So the problem can be undertaken without any reference to algebraic numbers at all -- and indeed, that is the approach we take in the present paper.

We note that the interesting feature of $A$ can be expressed in another way.  If $v_0, v_1, v_2$ are the rows of $A$, then there exist permutation matrices $\Pi_1, \Pi_2$ such that $\Pi_1 v_0 = v_1$ and $\Pi_2 v_0 = v_2$.  Then
\beq
0 = 3v_0 + 4v_1 + 5v_2 = (3 + 4\Pi_1 + 5\Pi_2) v_0.
\eeq

In particular, the matrix $3 + 4\Pi_1 + 5\Pi_2$ has a nontrivial kernel, and is thus singular.  It is easy to show (and we do so in Proposition~\ref{pr:smythcriterion} below) that the existence of a relation $\sum a_i x_i = 0$ between Galois conjugates is equivalent to the existence of a set of $n$ permutation matrices $\Pi_1, \dots, \Pi_n$ such that $\det(\sum a_i \Pi_i) = 0$; or, equivalently, of a set of $n-1$ permutation matrices such that $\sum_{i=1}^{n-1} a_i \Pi_i$ has $-a_n$ as an eigenvalue.  It is hard to believe that the question of what the eigenvalues of a fixed linear combination of permutation matrices can be has not been considered in the past, but we were unable to find any prior work (with the notable exception of \cite{speyer:moanswer}, which was a response to our asking about whether this problem had been studied!).

Smyth's conjecture asserts that a relation could occur between Galois conjugates precisely under a set of local conditions he showed to be necessary.  He demonstrated experimentally that Galois conjugates satisfying a given relation could be found in every case he tested where his conjecture predicted success.  But the number fields that arose can be of rather large degree; for example, Smyth's solutions to $17 \theta_1 + 19 \theta_2 + 29 \theta_3 = 0$ are over an extension whose Galois group is the symmetric group $S_{425}$.  It was and is hard to imagine that such examples come from any straightforward construction on the coefficients of the relation.

There has been limited progress towards proving Smyth's Conjecture; much of the related work has involved studying the circumstances under which a specific linear (or multiplicative) relation is satisfied by Galois conjugates.   For instance, in \cite{DJ15}, Dubickas and Jankauskas classified the irreducible polynomials of degree at most 8 such that a subset of their roots satisfy $\theta_1 + \theta_2 = \theta_3$ or $\theta_1 + \theta_2 + \theta_3 = 0$. 

In \cite{Dix97}, Dixon studied the relationship between the Galois action on the roots of an irreducible polynomial and the existence of a \textit{multiplicative} relation among the roots. Dixon generalized Smyth's observation that for such a multiplicative relation to exist, the associated Galois group cannot be the full symmetric group. 

The existence of a $\Q$-linear relation among $d$ Galois conjugates requires that the $\Q$-vector space spanned by the conjugates has dimension less than $d$. Another research direction has been to study how much less than $d$ the $\Q$-dimension of this vector space can be. This question is addressed in \cite{BDEPS04}, where the authors show that for almost all $n$, the largest degree of an algebraic number whose conjugates span a vector space of dimension $n$ is $2^n n!$. The authors additionally obtain sharp bounds when $\Q$ is replaced by a finite field or by a cyclotomic extension of $\Q$.

In the present paper, we prove that Smyth's conjecture is correct.  Our main theorem, as we warned, makes no reference to algebraic numbers and Galois conjugacy; we are proving Smyth's combinatorial formulation of the problem.

\begin{thm}
\label{th:main}
Let $a_1, \ldots, a_r$ be a set of elements of $\Q$.  Then there exists a probability distribution $(X_1, \ldots, X_r)$ on $\Q^r$, supported on a finite subset of the hyperplane $\sum a_i x_i = 0$, and whose marginals $X_i$ are equal in distribution, if and only if the following local conditions hold:
\begin{itemize}
\item For every prime $p$, and for every $i$,
\beq
|a_i|_p \leq \max_{j \neq i} |a_j|_p
\eeq
\item In the real absolute value,
\beq
|a_i| \leq \sum_{j \neq i} |a_j|
\eeq
for every $i$.

\end{itemize}
\end{thm}

We note that there is a natural way to generalize Smyth's conjecture from $\Q$ to general global fields~\cite[Conjecture 10.1]{HY22}.  We will prove in Theorem~\ref{th:ffl2g} that the analogue of Smyth's conjecture holds for all function fields of curves over finite fields, generalizing (and very much inspired by) the main result of \cite{HY22}, which proves this fact for $K = \F_q(t)$.

We now explain the combinatorial equivalence that makes Theorem~\ref{th:main} a proof of Smyth's conjecture. The equivalence between (1) and (2) below is drawn from Smyth's paper; criterion (3) is the relation with linear combinations of permutation matrices advertised above.

\begin{prop} Let $K$ be a global field and let $a_1, \ldots, a_r$ be a set of elements of $K$.  The following conditions are equivalent:
\begin{enumerate}
\item There exists a finite Galois extension $L/K$ and Galois conjugate elements $\theta_1, \ldots, \theta_r$ of $L$ satisfying $\sum a_i \theta_i = 0$.
\item There exists a probability distribution $(X_1, \ldots, X_r)$ on $K^r$, supported on a finite subset of the hyperplane $\sum a_i x_i = 0$, such that the marginals $X_i$ are all equal in distribution.
\item There exists an $N$ and a set of $r$ permutation matrices $\Pi_1, \ldots, \Pi_r \in S_N$ such that $\det(\sum a_i \Pi_i) = 0$.
\end{enumerate}
\label{pr:smythcriterion}
\end{prop}

\begin{proof}
    $(1) \Rightarrow (3)$: Without loss of generality, we will assume that $\theta_1, \dots, \theta_r$ comprise a complete Galois orbit, as if not, we can extend the argument to the full orbit, setting additional $a_i$ equal to 0. 
    
    Let $v_1, \dots, v_d$ be a linear basis for $L$ over $K$. Then there are coefficients $c_{ij} \in K$ such that $\theta_i = \sum_{j=1}^d c_{ij} v_j$. From the hypothesized linear relation, we have
    \begin{align*}
        \sum_{i=1}^r a_i \theta_i &= 0 \\ \sum_{i=1}^r a_i \Big(\sum_{j=1}^d c_{ij} v_j \Big) &= 0 \\ \sum_{i=1}^r \sum_{j=1}^d a_i c_{ij} v_j &= 0
    \end{align*}
    By linear independence of the $v_j$'s, it follows that each coefficient, $\sum_{i=1}^r a_i c_{ij} = 0$. 
    
Additionally, applying any $\sigma \in \Gal(L/K)$ to the assumed equation, we have $\sum_{i=1}^r a_i \sigma \theta_i = \sum_{i=1}^r a_i \theta_{\sigma(i)} = 0$, where we have identified $\sigma$ with its permutation action on the subscripts of $\theta_1, \dots, \theta_r$. Consequently, we have $\sum_{i=1}^r a_i c_{\sigma(i)j} = 0$ for all $\sigma \in G:=\Gal(L/K)$. 

Define the matrix $C:= (c_{ij})_{i \in [r], j \in [d]}$ and fix an enumeration of the elements of $G = \{\sigma_1, \dots, \sigma_s\}$. Let $\sigma C$ be the matrix $(c_{\sigma(i)j})_{i \in [r], j \in [d]}$ for any $\sigma \in G$. Now consider the block $sd \times r$ matrix $D$ constructed by vertically stacking $\sigma_1 C, \dots, \sigma_s C$. Since $G$ acts transitively on $\theta_1, \ldots, \theta_r$, the columns of $D$ lie in the same permutation orbit. And as already demonstrated, the row vectors of $D$ all lie in the hyperplane $\sum a_i x_i = 0$. As a result, we can let $N:= sd$, $v \in \R^N$ be any column of $D$, and then choose permutation matrices $\Pi_1, \dots, \Pi_r$ such that $\Pi_i v$ is the $i^{th}$ column of $D$. It then follows that $v \in \ker(\sum a_i \Pi_i)$, which implies the desired conclusion.

$(2) \Rightarrow (1)$: Write $S$ for the finite subset of $K^r$ on which the hypothesized distribution is supported.  The space of probability distributions supported on $S$ can be identified with the space of vectors in $\R^{|S|}$ with nonnegative coordinates summing to zero, and the condition that the marginals are all equal in distribution is a set of linear conditions on this space.  So the existence of the distribution $(X_1, \ldots, X_r)$ is equivalent to the statement that a certain affine linear subspace $L$ of $\R^{|S|}$ contains a vector $v$ with all coordinates nonnegative.  Since $L$ is defined over $\Q$, its rational points are dense in its real points. Choosing a rational point sufficiently close to $v$, we have a joint distribution $(X_1, \ldots, X_n)$ supported on $S$ which has all marginals equal and which furthermore assigns rational probability to each point in $S$.

  Having done so, we may take $C = \{c_{ij}\}_{1 \leq i \leq r, 1 \leq j \leq d}$ to be a multiset containing each point in $S$ with multiplicity proportionate to the probability that $(X_1, \dots, X_r)$ takes that value. Then the marginals $X_i$ all being equal in distribution means that the multisets $\{c_{ij}\}_{j=1}^d$ are independent of $i$. Since $K$ is a global field, there exists an $S_d$-extension $L$ of $K$.  Let $L_1$ be the fixed field of $S_{d-1}$ acting on $L$, let $\gamma_1$ be an algebraic generator for $L_1$, and let $\gamma_1, \dots, \gamma_d$ be the orbit of $\gamma_1$ under $S_d$, so that $S_d$ acts on this set of conjugates by its usual action on a set of $d$ letters.  Now define $\theta_i:= \sum_{j=1}^d c_{ij} \gamma_j$. By construction, the coefficient vectors $(c_{ij})_{j=1}^d \in K^d$ lie in the same permutation orbit, and so $\theta_1, \dots, \theta_r$ are Galois conjugates. Moreover, since $(c_{1j}, \dots, c_{rj}) \in K^r$ lies in the hyperplane $\sum a_i x_i =0$ for all $j$, we have
\begin{align*}
    \sum_{i=1}^r a_i \theta_i &= \sum_{i=1}^r a_i \Big(\sum_{j=1}^d c_{ij} \gamma_j\Big) \\ &= \sum_{j=1}^d \sum_{i=1}^r (a_i c_{ij}) \gamma_j \\ &=0,
\end{align*}
since each coefficient $\sum_{i=1}^r (a_i c_{ij}) =0$, by construction of the $c_{ij}$. 

$(3) \Rightarrow (2)$: Since the linear system is defined over $K$, there exists a $K$-valued vector $v \in \ker(\sum a_i \Pi_i)$. Let $X = (X_1, \dots, X_r)$ be the probability distribution on $K^r$ which draws a row uniformly at random from the $N \times r$ matrix $\Big[ \Pi_1 v \; | \dots | \; \Pi_r v \Big]$. The assumed equation $\sum a_i \Pi_i v = 0$ implies that $X$ is (finitely) supported on the hyperplane $\sum a_i x_i =0$. And because the $\Pi_i$ are permutations, each marginal distribution $X_i$ is a uniform draw from the coordinates of $v$.
\end{proof}

We now record the statement of Smyth's conjecture.
\begin{thm}
    Let $a_1, \ldots, a_r$ be a set of elements of $\Q$. There exists a finite Galois extension $L/\Q$ and Galois conjugate elements $\theta_1, \ldots, \theta_r$ of $L$ satisfying $\sum a_i \theta_i = 0$ if and only if the following local conditions hold:
\begin{itemize}
\item For every prime $p$ of $\Q$
\beq
|a_i|_p \leq \max_{j \neq i} |a_j|_p
\eeq
for every $i$.
\item In the real absolute value,
\beq
|a_i| \leq \sum_{j \neq i} |a_j|
\eeq
for every $i$.
\end{itemize}
    \label{th:smyth}
\end{thm}

\begin{proof}
This is immediate from Theorem~\ref{th:main} and the equivalence of (1) and (2) in Proposition~\ref{pr:smythcriterion}.
\end{proof}

\begin{rem}
    It is natural to wonder whether a certain linear relation among Galois conjugates is still possible if we impose constraints on the Galois group.  The methods used here do not seem to have much to say about this case.
\end{rem}

\begin{rem} The main theorem of this paper tells us which integers can appear as an eigenvalue of a linear combination $\sum_{i=1}^n a_i \Pi_i$ of permutation matrices with specified integer coefficients $a_i$.  Proving the theorem for arbitrary number fields would give a characterization of all algebraic integers which could be eigenvalues of such a linear combination.  In the case of the sum of two permutations (i.e.  $a_1 = a_2 = 1$) this was carried out by Speyer in \cite{speyer:moanswer}. Combining this remark with the previous one, one might ask:  if $G$ is a finite group and $a_1, \ldots, a_n \in \Z$ a list of coefficients, we may consider the set of elements of the group ring $\Z[G]$ of the form $\sum a_i g_i$ with $g_i \in G$.  Each such element can be considered as a linear endomorphism of $\Z[G]$ and as such has a set of eigenvalues.  For any fixed group $G$, this yields a finite set of algebraic integers.  What can we say about the union of this set as $G$ ranges over all groups in a class of interest?  The main theorem of the present paper shows that, if $G$ is allowed to range over all symmetric groups (or, equivalently, over all finite groups), then this set of eigenvalues includes all integers which satisfy Smyth's local triangle inequalities when appended to $a_1, \ldots, a_n$.  Yet another way of interpreting our theorem is as a statement about group rings of free groups.  If $F_2$ is freely generated by elements of $e_1$ and $e_2$, for instance, our result characterizes those elements $x = a+be_1+ce_2$ of $\Q[F_2]$ which become noninvertible in $\Q[G]$ for some finite quotient $G$ of $F_2$.  What if we ask about more complicated elements of the group ring, like $a + be_1 + ce_2 + d e_1 e_2$?  What if we ask this question where the free group is replaced by other discrete groups of interest?  This begins to resemble questions that arise in topology, where $M$ is a manifold with fundamental group $\pi$, and the rational homology of a finite cover of $M$ with Galois group $G$ is described in terms of some exact sequence of $\Q[G]$-modules whose coefficients are the projections to $\Q[G]$ of fixed elements of $\Q[\pi]$. In situations like this one often wants to know where there exists a finite $G$ where the rank of this homology group does something non-generic.
\end{rem}

\begin{rem} While we concentrate on additive relations among Galois conjugates here, Smyth shows (when $K=\Q$) that the conditions in Proposition \ref{pr:smythcriterion} are also equivalent to the existence of a multiplicative relation $\prod_i \theta_i^{a_i}$ among Galois conjugates.  It would be interesting to investigate what one can say about, e.g., multiplicative relations among Weil numbers, which in turn provide additive relations among zeroes of L-functions over finite fields.  See \cite{angleranks} for some recent work in this direction.
\end{rem}

\subsection{Outline of the proof}
In the remainder of this section, we explain how Smyth's conjecture can be seen as a statement of a ``non-deterministic local-to-global principle."  The proof could be carried out in full without introducing this notion, but we found this viewpoint to be decisive in allowing us to develop an approach to the problem.  In section~\ref{s:local}, we show that Smyth's conditions guarantee local solubility of the Smyth problem -- here local is meant in the non-deterministic sense introduced here, and does {\em not} have anything to do with Galois conjugates in finite extensions of local fields!  In section~\ref{s:ffl2g}, we prove Smyth's conjecture over function fields using an algebro-geometric argument; in this setting, we show that one can very directly piece together local solutions into a global solution.  In section~\ref{s:approxl2g}, we copy this approach over $\Q$, and find that, thanks to our eternal nemesis the archimedean place, the local solutions piece together only into an {\em approximate} global solution; in the combinatorial language, we find a matrix whose rows satisfy the desired equation and whose columns are {\em almost} permutations of one another.

\begin{figure}[h]
    \centering
    \includegraphics[width=0.8\textwidth]
    {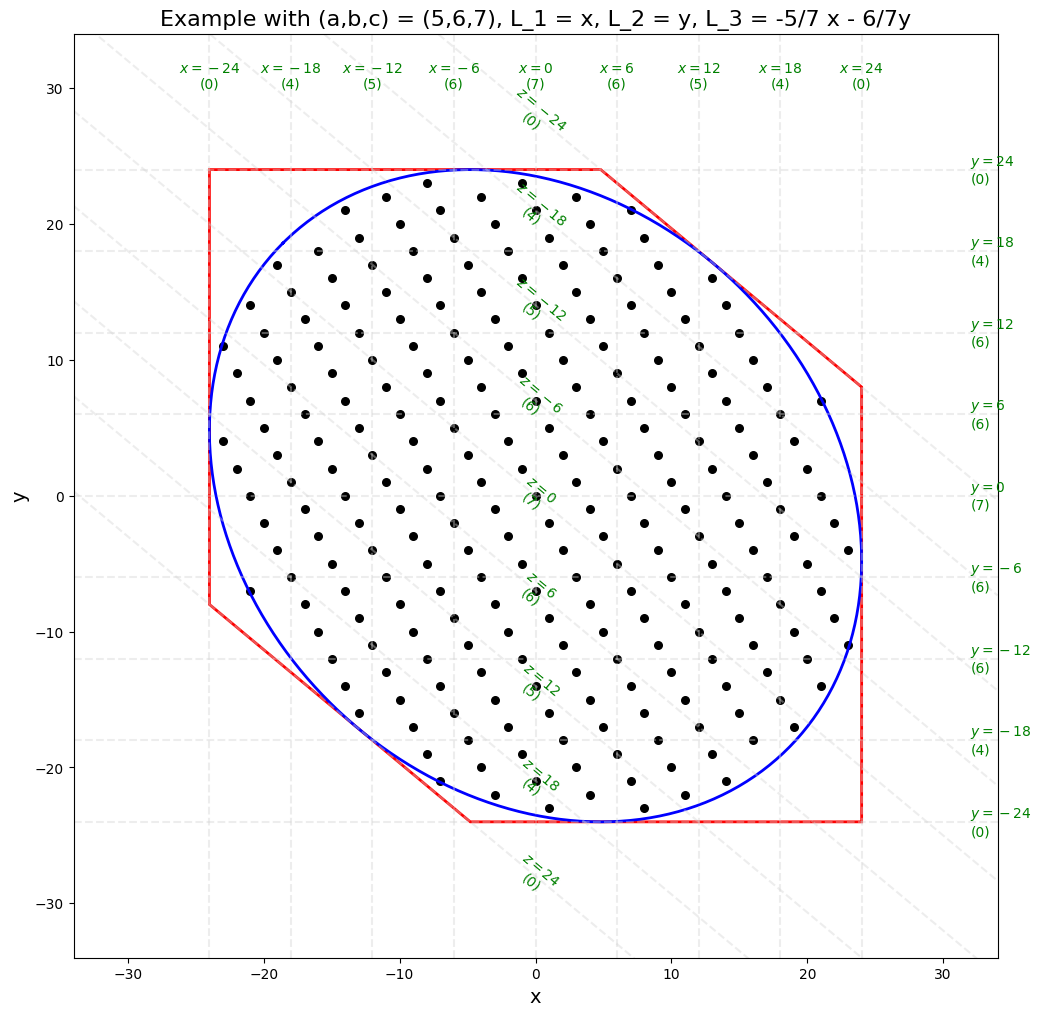}
    \caption{An approximate solution for $5L_1+6L_2+7L_3=0$}
    \label{fig:ellipse}
\end{figure}

Figure~\ref{fig:ellipse} depicts the nature of the construction.  We have three linear forms $L_1(x,y) = x, L_2(x,y) = y,$ and $L_3(x,y) = -(5/7)x - (6/7)y$ on the plane.  These linear forms satisfy the relation $5L_1 + 6L_2 + 7L_y = 0$. So each of the lattice points $\lambda$ in the inscribed ellipse corresponds to a triple $(L_1(\lambda), L_2(\lambda), L_3(\lambda))$ satisfying the same relation.  These triples form the rows of our matrix. 
 For the columns to be permutations of each other, we would need the multisets $L_1(\lambda), L_2(\lambda), L_3(\lambda)$ to agree.  But they do not -- for instance, as the diagram shows, there are $6$ lattice points in the ellipse with $L_2(\lambda) = 12$ but only $5$ with $L_1(\lambda) = 12$.  However, they {\em almost} agree; in fact, these numbers never differ by more than $1$.  The main result of this section, Proposition~\ref{pr:approxbalanced}, gives precise error bounds for this approximate solution.

 The final part of the paper is the longest, and is primarily additive-combinatorial in nature.  We recast the problem as one of finding a {\em balanced weighting} on the (hyper)edges of a finite hypergraph.  In the non-hyper case of directed graphs, this comes down to finding subgraphs of a directed graph in which every vertex has in-degree and out-degree equal, a well-understood problem.  The analogous problem for hypergraphs appears to be more subtle; but in this case, we are able to show that an approximately balanced weighting (the output of Proposition~\ref{pr:approxbalanced}) can always be perturbed to an actually balanced weighting, which gives the desired result.

\subsection{Non-deterministic local-to-global principles}

As we explained in the previous section, Smyth reduced the problem of studying additive and multiplicative linear relations among Galois conjugates to a question about the existence of certain finitely supported probability distributions on $\Q$.  In this section, we explain how this can be recast in terms of a novel class of Diophantine problems, which we call non-deterministic systems of equations.

By a {\em non-deterministic system of equations} over a ring $R$ we mean a collection of assertions about a joint distribution on $R$-valued random variables $X_1, \ldots, X_n$ which take one of the following two forms:
\begin{itemize}
\item $F(X_1, \ldots, X_n)$ and $G(X_1, \ldots, X_n)$ are equal in distribution with $F$, $G$ polynomials in $R[x_1,\ldots,x_n]$.  We denote such a relation $F(X_1, \ldots, X_n) \eqd G(X_1, \ldots, X_n)$.
\item $F(X_1, \ldots, X_n) \eqd Y$ for some specified random variable $Y$.
\end{itemize}

We say a non-deterministic system of equations is {\em homogeneous} if the polynomials $F$ and $G$ above are homogeneous and if all equations of the form $F(X_1, \ldots, X_n) \eqd Y$ have $Y=0$ (by which we mean the $Y$ is a random variable which is $0$ with probability one.) 

In the present paper, $R$ will always be a topological ring (though this topology might be the discrete topology!)  By a {\em solution} to a non-deterministic system of equations over $R$ we mean a joint distribution satisfying all the assertions and which has {\em compact support.}  (In case the topology on $R$ is discrete, this means finite support.)

If $X_1,\ldots, X_n$ is a solution to a homogeneous non-deterministic system of equations, then so is $\lambda(X_1, \ldots, X_n)$ for any $\lambda \in R$, and in particular the atomic distribution supported at the origin; we will typically neglect this trivial solution and, by analogy with the usual treatment of equations in projective space, understand a solution to a homogeneous non-deterministic system to mean a nonzero solution up to scaling.

\begin{exmp}
We present a few examples to help explain the scope of the definition.
\begin{itemize}
\item  The homogeneous system $X \eqd Y+Z, X+Y-Z=0$ has a solution in the traditional sense given by the point $(1:0:1) \in \P^2(R)$; thus it has a solution as a non-deterministic system by any joint distribution in which $X=Z$ and $Y=0$ with probability $1$.  More generally, any system of equations with a deterministic solution has a non-deterministic one given by an atomic distribution supported at the deterministic solution.
\item The equation $X \eqd -X$ is a homogeneous non-deterministic equation in one variable.  Any symmetric random variable $X$ is a solution, but there is no nontrivial deterministic solution.
\item The condition of compact support might seem at first to be somewhat arbitrary.  We impose it, first of all, because it is essential in the application to Smyth's conjecture. Relaxing the condition of compact support makes many more equations solvable, perhaps too many. 
 For instance, if $X$ and $Y$ are independent draws from a standard L\'{e}vy distribution, then $Z = (1/4)(X+Y)$ is equal in distribution to $X$ and $Y$; so the equation $X \eqd Y \eqd Z, X+Y-4Z = 0$ is solvable over $\R$ in general distributions but not in compactly supported distributions (consider the variance of $Z$, or its maximum and minimum.)

\item If $U$ is a random variable uniformly distributed on $[0,1]$, solutions to the non-homogeneous system $X_1 \eqd X_2 \eqd \ldots \eqd X_n \eqd U$ are called {\em copulas}.  There has been some interest in the probability literature in copulas with low-dimensional support; in some sense, the probability distributions in this paper, which are joint distributions on $n$ variables supported on a linear subspace, are in the same vein.

\end{itemize}
\end{exmp}

Smyth's conjecture can then be rephrased as the assertion that, under certain local conditions on $a_1, \ldots, a_n \in \Q$, the homogeneous non-deterministic system of equations
\begin{equation}
X_1 \eqd \ldots \eqd X_n, \sum a_i X_i = 0
\label{eq:ndsmyth}
\end{equation}
has a nontrivial solution over $\Q$.  (To be more precise: we need a nontrivial solution which has finite support.)  But the fact that the conditions are local should be suggestive.  In fact, the nature of the proof of Theorem~\ref{th:main} is as follows:  we first prove that the local condition at $v$ guarantees that \eqref{eq:ndsmyth} has a solution over $\Q_v$.  Just as in the deterministic situation, the study of local solutions is much easier than that of global solutions.  We then show that the existence of solutions to \eqref{eq:ndsmyth} over every completion $\Q_v$ implies the existence of a solution over $\Q$.  In other words, we prove a local-to-global principle for \eqref{eq:ndsmyth}.

This raises the natural question:  what types of nondeterministic equations admit a local-to-global principle?  In the classical Diophantine setting this is a rich problem, with many questions remaining the subject of active research even today.  In the nondeterministic setting, even the case of {\em linear} systems of equations, as treated in this paper, have real content.  In particular, we do not know the answer to the following question.

\begin{question}
Let $\set{L_{i;j}}$ be a colletion of linear forms in variables $X_1, \ldots X_r$ over a global field $K$.  Under what additional hypotheses can one guarantee that a system of linear non-deterministic equations
\begin{eqnarray*}
L_{1;1} \eqd L_{2;1} \eqd \ldots \eqd L_{n_1;1} \\
L_{1;2} \eqd L_{2;2} \eqd \ldots \eqd L_{n_2;2} \\
\ldots \\
L_{1;m} \eqd L_{2;m} \eqd \ldots \eqd L_{n_m;m} \\
\end{eqnarray*}
has a compactly (i.e. finitely) supported solution over $K$ if and only if it has a compactly supported solution over $K_v$ for all places $v$?
\end{question}

 Our main result, Theorem~\ref{th:main}, which is essentially equivalent to Smyth's conjecture, is that this local-to-global principle holds when $m=1$, $K=\Q$ and $r=n-1$.  This involves a change of variables: choosing a basis for the hyperplane $\sum a_i x_i = 0 \in \Q^n$ is the same thing as choosing $n$ linear forms $L_1, \ldots, L_n$ on $\Q^{n-1}$ satisfying the relation $\sum a_i L_i = 0$.  Furthermore, the condition that a distribution supported on the hyperplane has all marginals equal is equivalent to the condition that the corresponding distribution on $\Q^{n-1}$ satisfies $L_1 \eqd \ldots \eqd L_n$.  This is just another way of saying that the non-deterministic equation $L_1 \eqd \ldots \eqd L_n$ on $\Q^{n-1}$ is equivalent to the system of non-deterministic equations $X_1 \eqd \ldots \eqd X_n, \sum a_i X_i \eqd 0$ on $\Q^n$.  In the last equation, the $\eqd$ is optional; to require a linear form to be equal in distribution to the atomic variable $0$ is just to constrain that linear form to be $0$ (up to probability zero events, which do not concern us.)
 
 We will also show in Theorem~\ref{th:ffl2g} in the following section that this local-to-global principle always holds when $m=1$ and $K$ is the function field of a smooth proper curve over a finite field.  

There are many further questions we can ask, even about linear systems:
\begin{itemize}
\item Can one go beyond the Hasse principle and ask about weak approximation; that is, are the global solutions in some sense dense, or even in some sense equidistributed, in the adelic solutions?
\item One can imagine a broader notion of nondeterministic system of equations:  what if we require some subsets of the variables to be independent?  Or what if we allow equations enforcing equality of joint distributions of two different subsets of variables? 
\item Of course, one might also ask about non-deterministic equations of higher degree.  We have very little idea of what to expect in this case.
\end{itemize}

\begin{rem} There is not \textit{always} a local to global principle for systems of linear equations, as the following example shows; we learned this example in a related context from David Speyer in \cite{speyer:moanswer}.  Let $\alpha$ be an integer in a quadratic imaginary field $K$ which has complex magnitude $2$ and is not twice a root of unity -- for instance, we can take $\alpha = (1/2)(3+\sqrt{-7})$ -- and consider the equations
\beq
X \eqd Y \eqd Z, X+Y-\alpha Z = 0
\eeq
or, equivalently,
\beq
X \eqd Y \eqd \alpha^{-1}(X+Y).
\eeq

For any non-archimedean completion $K_v$ of $K$, this system is solved by the joint distribution where $X$ and $Z$ are independently drawn from $\OO_{K_v}$ and $Y$ is given by $\alpha Z - X$.  And over $\C$, we may take $X$ to be drawn uniformly from the unit circle, $Y$ to be $X$, and $Z$ to be $2 \alpha^{-1} X$.  On the other hand, there is no solution over $K$, as we now show.  We note first that any complex solution (and thus any solution over $K$) has $X = Y$.  For
\beq
4 \EE Z\bar{Z} = |\alpha^2| \EE Z\bar{Z} = \EE (X+Y)(\bar{X}+\bar{Y}) = \EE X\bar{X} + \EE Y\bar{Y} + 2 \EE \Rea X \bar{Y}.
\eeq
But by the hypothesis that $X \eqd Y \eqd Z$, we have $\EE X \bar{X} = \EE Y \bar Y = \EE Z \bar Z$, so we have
\beq
\EE X \bar{X} = \EE \Rea X \bar{Y}.
\eeq
On the other hand,
\beq
\EE (X-Y)(\bar{X} - \bar{Y}) = 2(\EE X \bar{X} - \EE \Rea X \bar{Y}) = 0
\eeq
But the left-hand side is the expected squared magnitude of $X-Y$, so it can be $0$ if and only if $X-Y$ is $0$ with probability $1$.

  But then our equations reduce to a requirement that $X$ and $Z = (2/\alpha)X$ be equal in distribution, which means in particular that the finite support of $X$ is invariant under multiplication by $2/\alpha$, which is impossible since $2/\alpha$ has infinite order by hypothesis.

This example is clearly in some sense a ``boundary case," where the local conditions at $\C$ enforce a further algebraic condition $X=Y$ on the distribution.  This suggests that there might be a general nondeterministic local-to-global principle for linear systems of equations subject to some special hypotheses involving roots of unity; that these issues arise in the boundary cases of Smyth's conjecture over number fields  is already observed in \cite{HY22}.
\end{rem}

\subsection{Acknowledgments}
The authors are very grateful for many useful conversations about Smyth's conjectures over the years, with Jen Berg, Bobby Grizzard, Timo Sepp\"{a}l\"{a}inen, David Speyer, Betsy Stovall, Benny Sudakov, Sameera Vemulapalli, and John Yin, among others.  The first author's research is supported by NSF grant DMS-2301386 and by the Office of the Vice Chancellor for Research and
Graduate Education at the University of Wisconsin–Madison with funding from the Wisconsin Alumni Research Foundation. Some of the work in this paper was carried out while the first author was visiting the Simons Laufer Mathematical Sciences Institute.

\section{Non-deterministic solutions over local fields}

\label{s:local}

For the rest of this paper, we restrict our attention to homogeneous non-deterministic systems of linear equations of the form
\begin{equation}
L_1(X_1, \ldots, X_r) \eqd \ldots \eqd L_n(X_1, \ldots X_r)
\label{l1eqdln}
\end{equation}
over a local or global field $K$, with $r < n$.

We begin by distinguishing a certain distinguished class of non-deterministic equations over local fields.

\begin{defn} If $v$ is a place of $K$, we define an {\em ellipsoid} in $K_v^r$ to be:
\begin{itemize}
\item A lattice in $K_v^r$ if $v$ is non-archimedean (that is, a finitely generated $\OO_{K_v}$-submodule of $K_v^r$ which spans $K_v^r$ as a $K_v$-vector space;)
\item A set of the form $\set{x: \lVert x \rVert < M}$ where $M > 0$ is a real number and $\lVert \rVert$ is a Hermitian norm on $K_v^r$, when $v$ is a real or complex place.
\end{itemize}

If $\Omega$ is an ellipsoid in $K_v^r$, we define the {\em unitary group} $U(\Omega) \subset \GL_r(K_v)$ to be the stabilizer of $\Omega$, which is a compact subgroup.  When $v$  is non-archimedean, $U(\Omega)$ is commensurable with $\GL_r(\OO_{K_v})$.  When $v$ is archimedean, $U(\Omega)$ is the group of linear transformations preserving $\lVert \rVert_v$, which is an orthogonal group if $v$ is real and a unitary group in the usual sense if $v$ is complex.

We say a non-deterministic system of equations \eqref{l1eqdln} over $K_v$ is {\em ellipsoidal} if there exists an ellipsoid $\Omega$ such that $L_1, \ldots, L_n$ lie in a single orbit of $U(\Omega)$.  (Here, $U(\Omega)$ is acting on the $L_i$ by change of coordinates; in other words, we are using the dual action to the action of $U(\Omega)$ on $K_v^r$.)

\end{defn}

In the case $r=n-1$, one can explicitly describe necessary and sufficient conditions for the existence of local solutions to \eqref{l1eqdln}.  If $a_1, \ldots, a_n$ are elements of a local field $K_v$, we say that $a_1, \ldots, a_n$ {\em satisfy the triangle inequality} if, for all $i$,
\beq
|a_i|_v \leq \max_{j \neq i} |a_j|_v
\eeq
(when $v$ is non-archimedean) or
\beq
|a_i|_v \leq \sum_{j \neq i} |a_j|_v.
\eeq
(when $v$ is archimedean.)

\begin{prop}  
Let $K_v$ be a local field and let $a_1, \ldots, a_n$ be a set of elements of $K_v$, and let $L_1, \ldots, L_n$ be a tuple of linear forms in $X_1,\ldots, X_{n-1}$ such that $\sum a_i L_i = 0$ and which satisfy no other linear relation.  The following are equivalent:
\begin{enumerate}
\item $a_1, \ldots a_n$ satisfy the triangle inequality;
\item The equation $L_1 \eqd \ldots \eqd L_n$ is ellipsoidal;
\item The equation $L_1 \eqd \ldots \eqd L_n$ has a solution which is uniform on an ellipsoid in $K_v^{n-1}$, and which, when $v$ is non-archimedean and the largest coefficient of each $L_i$ is a $v$-adic unit, is uniform on $\OO_{K_v}^{n-1}$.
\item The equation $L_1 \eqd \ldots \eqd L_n$ has a solution.
\end{enumerate}
\label{pr:localmain}
\end{prop}

\begin{proof}
We begin with $(1) \Rightarrow (2)$, which is the most substantial part.

If $x$ is a vector in $K_v^{n-1}$, we denote by $||x||_v$ the standard Euclidean norm of $x$ when $v$ is real, the standard Hermitian norm when $v$ is complex, and the norm which sends $\ic{m}_v^r \OO_{K_v}^{n-1} \bs \ic{m}_v^{r+1} \OO_{K_v}^{n-1}$ to $|k_v|^{-r}$ when $v$ is non-archimedean.  Note that, for any $v$, we have that $||\lambda x||_v = |\lambda|_v ||x||_v$ for any $\lambda \in K_v$ and $x \in K_v^{n-1}.$

\begin{lemma}
    Let $K_v$ be a local field, and let $\ell_1, \ldots, \ell_m$ be elements of $\R$ lying in the image of the valuation of $K_v$. Furthermore, if $v$ is non-archimedean, suppose that for all $i$,
\beq
\ell_i \leq \max_{j \neq i} \ell_j
\eeq
and if $v$ is archimedean, suppose that, for all $i$, 
\beq
\ell_v \leq \sum_{j \neq i} \ell_v.
\eeq
    
    Let $d$ be an integer greater than $1$. Then there are vectors $b_1, \ldots, b_m$ in $K_v^d$ such that $||b_i||_v = \ell_i$ for all $i$ and $\sum_i b_i = 0$.
    \label{le:triangle}
\end{lemma}

\begin{proof}
    This is a generalized converse triangle inequality over an arbitrary local field, and surely must be standard, but not having been able to find a proof in the literature, we include one here.

    We argue by induction on $m$.  We first note that the cases $m=1$ and $m=2$ are easy, so we can start our induction with $m=3$ as the base case.
    
    We first consider the case where $v$ is archimedean.  The case $m=3$ is the usual (converse) triangle inequality which asserts that three real lengths $\ell_1, \ell_2, \ell_3$ satisfying the triangle inequality in fact are the edge lengths of some triangle in $\R^2$ (and a fortiori in $\C^2.$)  Now suppose $m$ is at least $4$.  We note that the given inequality can be written as $2 \max \ell_i \leq \sum \ell_i$.

    We may split $1, \ldots, m$ into two disjoint subsets $S$ and $T$, each of size at most $m-2$.  Write $\sum S$ for the sum of $\ell_i$ as $i$ ranges over $S$, and $\max S$ for the maximum of $\ell_i$ as $i$ ranges over $S$, and likewise for $T$.  We claim the intervals $[2 \max S - \sum S, \sum S]$ and $[2 \max T - \sum T, \sum T]$ intersect.  Suppose on the contrary they are disjoint.  Without loss of generality we may assume that $T$ is the set with the larger sum, and so $\sum S < 2 \max T - \sum T$.  But this means that $2 \max T > \sum S + \sum T$, the latter quantity being the sum of {\em all} the $\ell_i$, and this violates the hypothesis.  We conclude that the two intervals $[2 \max S - \sum S, \sum S]$ and $2 \max T - \sum T, \sum T]$ have nontrivial intersection, and indeed, since both intervals have nonnegative upper bound, there is a nonnegative real number $\ell$ contained in both intervals.  We denote the set $\set{\ell_i: i \in S} \cup \set{\ell}$ by $S \cup \ell$ for short. We claim that $S \cup \ell$ satisfies the generalized triangle inequality.  For
    \beq
    2 \max (S \cup \ell) - \sum (S \cup \ell)
    \eeq

    is either $2 \ell - \sum S - \ell = \ell - \sum S$, or $2 \max S - \sum S - \ell$.  Both are negative, since $\ell$ lies in $[2 \max S - \sum S, \sum S]$.  The same reasoning applies to $T \cup \ell$.  By applying the induction hypothesis to both $S \cup \ell$, there is choice of vectors $v_1, \ldots, v_m, v$ such that $||v_i|| = \ell_i, ||v|| = \ell, \sum_{\ell_i \in S} v_i = v$.  The same applies to $T$, with some other auxiliary vector $v'$ of length $\ell$; but applying a norm-preserving linear transformation to all $|T|+1$ vectors we can choose $v'=v$.  Now we have a set of vectors of the desired length indexed by $S$ and a set of vectors of the desired length indexed by $T$, and negating the latter yields the desired set of vectors for $\ell_1, \ldots, \ell_m$.

    In the non-archimedean case, an easier version of the same argument suffices.  Here, the generalized triangle inequality states that the maximum value of $\ell_i$ appears at least twice as $i$ ranges over $1, \ldots, m$. Again, $m=1,2$ is trivial and $m=3$ is standard and can be taken as the base case.  So suppose $m=4$.  We split $1,\ldots, m$ into disjoint subsets $S$ and $T$, each of size at most $m-2$, such that the maximum value $M$ of $\ell_i$ (which by hypothesis appears at least twice) appears at least once in $S$ and at least once in $T$.  Now the induction hypothesis tells us that there's a set of $|S|+1$ vectors summing to zero such that $||v_i|| = \ell_i$ for each $i \in S$ and the remaining vector $v$ has norm $M$.  The same holds with $S$ replaced by $T$, and just as in the non-archimedean case we can remove the auxiliary length $M$ vector, leaving us with the desired set of vectors $v_1, \ldots, v_m$ of the desired norms.
    \end{proof}

We now apply Lemma~\ref{le:triangle} with $\ell_i = |a_i|$.  Let $b_1,\ldots b_n$ be the vectors produced by Lemma~\ref{le:triangle} and write $x_i$ for $b_i/a_i$.  Then $\|x_i\|_v = 1$ for all $i$, and 
\beq
\sum_{i=1}^n a_i x_i = 0.
\eeq
This linear relation, and the fact that the $L_i$ satisfy no {\em other} relation than this one, imply that there exists a linear isomorphism $\gamma: K_v^{n-1} \ra K_v^{n-1;\vee}$ sending $b_i$ to $L_i$.  If $U$ is the stabilizer of the standard norm $\| \|_v$, then $b_1, \ldots, b_n$ lie in a single orbit of $U$.  This implies that $L_1, \ldots, L_n$ lie in a single orbit of $\gamma U \gamma^{-1}$, so $L_1 \eqd \ldots \eqd L_n$ is ellipsoidal, as claimed. 

$(2) \Rightarrow (3)$:  This implication does not use the running hypothesis $r=n-1$ so we include the proof for general $r < n$.  If \eqref{l1eqdln} is ellipsoidal over $K_v$, then a random variable $(X_1, \ldots, X_r)$ drawn uniformly from $\Omega_v$ is a solution to \eqref{l1eqdln} over $K_v$.  To see this, let $u_{i,j}$ be a transformation in $U(\Omega_v)$ taking $L_i$ to $L_j$; since the distribution of $(X_1, \ldots, X_{n-1})$ is invariant under $u_{i,j}$, the distribution of $L_j(X_1, \ldots, X_{n-1}) = L_i(u_{i,j}(X_1, \ldots, X_r))$ is the same as that of $L_i(X_1, \ldots, X_r)$.  If the largest coefficient of $L_i$ is a $v$-adic unit, and $X_1, \ldots, X_r$ is drawn uniformly from $\OO_{K_v}^r$, then $L_i(X_1,\ldots, X_r)$ is uniform in $\OO_{K_v}$, so indeed that choice of distribution suffices under that hypothesis on the $L_i$.

$(3) \Rightarrow (4)$: immediate.

$(4) \Rightarrow (1)$:  The necessity of $a_1, \ldots, a_n$ satisfying the triangle inequality was already observed by Smyth to be a necessary condition over $\Q$.  The proof in the local case is in essence the same, but we recall it here.   

Let $X_1, \ldots, X_{n-1}$ be a random variable satisfying $L_1 \eqd \ldots \eqd L_n$, and write $M$ for the maximum value of $|x|_v$ for $x$ in the support of $L_i(X_1, \ldots, X_{n-1})$.   To be more precise, ``maximum" has to be understood in the probabilistic sense: $M$ is the supremum of the set of real numbers $r$ such that $|L_i(X_1, \ldots, X_{n-1})|_v > r$ has positive probability.   Because the $L_i$ are all equal in distribution, $M$ does not depend on $i$.

Now suppose $v$ is archimedean and
\beq
|a_i| > \sum_{j \neq i} |a_j|.
\eeq
(The proof in the non-archimedean case is exactly the same, with $\max$ replacing $\sum$ everywhere, so we exclude it.)  Let $\delta > 0$ be a small real number; then, with positive probability, $|L_i| \geq M - \delta$.  But
\beq
0 = \sum a_i L_i = a_i L_i + \sum_{j \neq i} a_j L_j.
\eeq
The latter summand is bounded above by $M \sum_{j \neq i} |a_i|$ with probability $1$.  Making $\delta$ small enough that
\beq
 (M-\delta) |a_i| > M \sum_{j \neq i} |a_i|
\eeq
yields a contradiction.
\end{proof}

\section{The case of global function fields}

\label{s:ffl2g}

Before we proceed to the proof of our main theorem, we will show that a local-to-global principle holds for general systems of linear equations over function fields of curves over finite fields.  This proof is conceptually simple and shows that the real difficulties of the present paper are creatures of the archimedean places.  The argument here is very much indebted to that of \cite{HY22}, which proves Theorem~\ref{th:ffl2g} in the case $K = \F_q(t)$.

In what follows, we will refer to a ``solution" of a non-deterministic system of equations as shorthand for a ``compactly supported non-deterministic solution."

\begin{thm}
Suppose $K$ is the function field of a smooth proper curve over a finite field, and suppose
\begin{equation}
L_1(X_1, \ldots, X_r) \eqd \ldots \eqd L_n(X_1, \ldots, X_r)
\label{eq:lineq}
\end{equation}
is a non-deterministic system of linear equations over $K$, with $r < n$.  Then \eqref{eq:lineq} has a solution over $K$ if and only if it has a solution over the completion $K_v$ for every place $v$ of $K$.
\label{th:ffl2g}
\end{thm}

\begin{cor} 
    Let $K$ be the function field of a smooth proper curve over a finite field, and let $a_1, \ldots, a_n$ be a set of elements of $K$. There exists a finite Galois extension $L/K$ and Galois conjugate elements $\theta_1, \ldots, \theta_r$ of $L$ satisfying $\sum a_i \theta_i = 0$ if and only if, for every place $v$ of $K$ and every $i$ we have
\beq
|a_i|_v \leq \max_{j \neq i} |a_j|_v.
\eeq
  
\end{cor}

\begin{proof} (of Corollary) Let $L_1, \ldots, L_n$ be linear forms in $X_1, \ldots, X_{n-1}$ which satisfy $\sum a_i L_i = 0$ and no other linear relation.  
 Proposition~\ref{pr:localmain} tells us that, under the local hypotheses of this corollary, there is, for each place $v$, a probability distribution on $K_v^n$ supported on the hyperplane $\sum_{i=1}^n a_i x_i = 0$ and with all $n$ marginals equal.  We can choose a $K_v$-basis $y_1, \ldots, y_{n-1}$ of linear forms on the hyperplane such that $x_i = L_i(y_1, \ldots, y_{n-1})$ for $i=1,2,\ldots, n$.  The distribution guaranteed by  Proposition~\ref{pr:localmain} now provides a non-deterministic solution to the equation $L_1 \eqd \ldots\eqd L_n$ over $K_v$.  This being the case for every $v$, Theorem~\ref{th:ffl2g} implies that $L_1 \eqd \ldots\eqd L_n$ has a solution by a random variable $X = (X_1, \ldots, X_{n-1})$ valued in $K^{n-1}.$ The joint distribution on $L_1(X), \ldots, L_n(X)$ is then the probability distribution on $K^n$ whose existence is the desired conclusion of Theorem~\ref{th:main}.

\end{proof}

The rest of this section will be devoted to proving Theorem~\ref{th:ffl2g}.

\subsection{Non-deterministic solutions over non-archimedean local fields}

We begin by considering the local situation.   Suppose for this section that \eqref{eq:lineq} is a non-deterministic equation over $K_v$ (i.e there is no requirement here that its coefficients lie in $K$.)

Let $\Omega$ be a compact subset of $K_v^r$, and write $\Lambda(\Omega)$ for the $\OO_{K_v}$-submodule of $K_v^r$ spanned by $\Omega$.  Let $L = L(X_1, \ldots, X_r)$ be a linear form over $K_v$. Then $L(\Lambda(\Omega))$ is the $\OO_{K_v}$-submodule of $K_v$ spanned by $L(\Omega)$.  It follows from the non-archimedean triangle equality that
\beq
\max_{x \in \Lambda(\Omega)} |L(x)|_v = \max_{x \in \Omega} |L(x)|_v.
\eeq

So suppose that $(X_1, \ldots, X_r)$ is a compactly supported distribution on $K_v^r$ such that $L_1(X_1, \ldots, X_r) \eqd \ldots \eqd L_n(X_1, \ldots, X_r)$.   Let $\Omega \in K_v^r$ be the support of $(X_1, \ldots, X_r)$.  The hypothesized equalities of distribution imply that $\max |L_i(X_1, \ldots, X_r)|_v$ is the same for all $i$; equivalently, $\max_{x \in \Omega} |L_i(x)|_v$ is the same for all $i$.  And by the discussion in the previous paragraph, this means that $\max_{x \in \Lambda(\Omega)} |L_i(x)|_v$ is the same for all $i$.

\subsection{Non-deterministic solutions over global function fields}

Now suppose $K$ is the function field of a smooth proper curve $C$ over a finite field $k = \F_q$, and suppose that $\eqref{eq:lineq}$ is an equation over $K$ which has a non-deterministic solution over $K_v$ for each place $v$ of $K$.  For each $v$, let $\Omega_v \subset K_v^r$ be the support of the local solution $(X_1, \ldots, X_r)$, and let $\Lambda_{0,v}$ be the $\OO_{K_v}$-submodule of $K_v^r$ spanned by $\Omega_v$.  It will be more convenient for us if $\Lambda_{0,v}$ is a {\em lattice} (an $\OO_{K_v}$-module whose $K_v$-linear span is all of $K_v^r$) so we define $\Lambda_v$ to be $\Lambda_{0,v} + m^N \OO_{K_v}^r$, where $m$ is a uniformizer in $\OO_{K_v}$ and $N$ is large enough so that $\max_{x \in \Lambda_v} |L_i(x)|_v = \max_{x \in \Lambda_{0,v}} |L_i(x)|_v$ for all $i$.

 We now have, for each place $v$, a lattice $\Lambda_v \in K_v^r$ such that $\max_{x \in \Lambda_v} |L_i(x)|_v$ is the same for all $i$. For all places $v$ such that all the coefficients of $L_1, \ldots, L_n$ are $v$-adic units (in particular, for all but finitely many places) we can and do take $\Lambda_v = \OO_{K_v}^n$. 

Now let $\FF$ be the sheaf on $C$ whose sections are those elements of $K^r$ whose image in $K_v^r$ lies in $\Lambda_v$ for every $v$.  We have chosen the $\Lambda_v$ such that, for every $v$, the lattices $L_1(\Lambda_v), \ldots, L_n(\Lambda_v) \in K_v$ are all equal.  It follows that the subsheaves $L_1(\FF), \ldots, L_n(\FF)$ of $K$ are all equal, since they are determined by the same local conditions.  We denote this line bundle by $\LL$.  Each linear form now provides a surjective morphism of sheaves $L_i: \FF \ra \LL$.  Denote the kernel of this map by $\KK_i$.  Choose some ample line bundle $\LL_0$ on $C$ and for any other sheaf $\SSS$ on $C$ denote $\SSS \tensor_{\OO_C} \LL_0^{\tensor d}$ by $\SSS(d)$.  Note that $\FF(d)$ and $\LL(d)$ can still be identified with subsheaves of $K^r$ and $K$ respectively.

We will now show that, for some sufficiently large $d$, taking $(X_1, \ldots, X_r)$ to be drawn uniformly from $\Gamma(C,\FF(d)) \subset K^r$ affords the desired non-deterministic solution to \eqref{eq:lineq} over $K$.  

For each $i$, the exact sequence
\beq
0 \ra \KK_i(d) \ra \FF(d) \ra \LL(d) \ra 0
\eeq
provides an exact sequence
\beq
\Gamma(C,\FF(d)) \ra \Gamma(C,\LL(d)) \ra H^1(C,\KK_i(d))
\eeq
But $H^1(C,\KK_i(d))$ vanishes for $d$ large enough by Serre's vanishing theorem, so $\Gamma(C,\FF(d)) \ra \Gamma(C,\LL(d))$ is a surjection of finite groups.  A surjective homomorphism of finite groups projects uniform distribution onto uniform distribution.  We have thus shown that if $(X_1, \ldots, X_r)$ is drawn uniformly from the finite set $\Gamma(C,\FF(d)) \subset K^r$, we have, for every $i$, that $L_i(X_1, \ldots, X_r)$ is uniformly distributed on $\Gamma(C,\LL(d)) \in K$.  So $(X_1,\ldots, X_r)$ is the desired non-deterministic solution of \eqref{eq:lineq} over $K$.

\section{Approximate local-to-global for non-deterministic solutions over $\Q$}

\label{s:approxl2g}

In the global setting, we say an equation \eqref{l1eqdln} over $K$ is ellipsoidal if it is ellipsoidal over $K_v$ for all places $v$; in this case we denote the ellipsoid witnessing the condition at $v$ by $\Omega_v$.  Note that, if all coefficients of some $L_i$ are $v$-adic units, and $X_1, \ldots, X_r$ are drawn independently from $\OO_{K_v}^r$, then $L_i(X_1, \ldots, X_r)$ is uniformly distributed in $\OO_{K_v}$.  So if all the coefficients of {\em all} the $L_i$ are $v$-adic units, \eqref{l1eqdln} is ellipsoidal with $\Omega_v = \OO_{K_v}^r$.  This implies in particular that every system \eqref{l1eqdln} over $K$ is locally ellipsoidal for all but finitely many $v$, with $\Omega_v$ taken to be the standard ellipsoid $\OO_{K_v}^r$.

From this point on, we restrict to the case $K=\Q$, since that is the case we will confine ourselves to for the remainder of the paper. Our expectation is that some version of the natural analogue of Proposition~\ref{pr:approxbalanced} will hold over an arbitrary number field $K$.

The main point of this section is to show that when \eqref{l1eqdln} is ellipsoidal, the local solutions over $\Q_p$ and $\R$ can be put together into an {\em approximate} solution over $K$, in a sense that we make precise below in Proposition~\ref{pr:approxbalanced}.

First of all, things will be simpler if we take a little care to make our ellipsoidal structure compatible with the global field.  At the non-archimedean places, there is no issue; we are just specifying a lattice in $\Q_p^r$, and any such lattice is ``global in origin" in the sense that it is the closure of the image of some lattice in $\Q^r$.  Over $\R$ the situation is not as good -- we might have perversely chosen a Euclidean norm on $\R^r$ which does not arise from any norm over $\Q$.  Fortunately, we can bounce back from any poor choices of this kind we may have made. 
By hypothesis, there exists a positive definite quadratic form $Q_0$ on $\R^r$ such that $L_1, \ldots, L_n \in (\R^r)^\vee$ lie in a single orbit of the orthogonal group $U_{Q_0}$ of $Q_0$.  The form $Q_0$ can be seen as an invertible linear isomorphism $\phi:\R^r \ra (\R^r)^\vee$.  With this notation, we can define the dual quadratic form $Q_0^\vee$ by
\beq
Q_0^\vee(L) = Q_0(\phi^{-1}(L)).
\eeq
The orthogonal group of $Q_0^\vee$ (in the dual action of $\GL(\R^r)$ on $(\R^r)^\vee$) is the same as that of $Q_0$.  So $L_1, \ldots, L_n$, being in a single orbit of that group, have the property that $Q_0^\vee(L_i) = Q_0^\vee(L_j)$ for all $i,j$.

We now consider the space of {\em all} positive definite quadratic forms $P$ on $(\Q^r)^\vee$ which assign the same value to each $L_i$.  For any such form, the $L_i$ lie in a single orbit of the orthogonal group of $P$.  The quadratic forms on $(\Q^r)^\vee$ form a vector space over $\Q$, and for each $i,j$, the constraint $P(L_i) = P(L_j)$ is a linear condition on that space.  So we have a linear subspace of the space of quadratic forms cut out by linear relations over $\Q$, and $Q_0^\vee$ provides an example of a {\em real} point of that subspace which is positive definite.  This implies that the subspace also contains $\Q$-rational forms which are positive definite.  We take $P$ to be one such, and let $Q$ be its dual.  Now $Q$ is a quadratic form on $\Q^r$ with the property that $L_1, \ldots, L_n$ all lie in the same orbit of the orthogonal group $O_Q(\R) = O_P(\R)$.  But we observe that, by Witt's theorem, $L_1,\ldots,L_n$ are indeed in an orbit of $O_Q(\Q) = O_P(\Q)$.

\begin{prop} Let $K = \Q$. Suppose the system
\beq
L_1(X_1, \ldots, X_r) \eqd \dots \eqd L_n(X_1, \ldots, X_r)
\eeq
is ellipsoidal, with $\Omega_p = \Z_p^r$ for all but finitely many primes $p$, and $\Omega_\R$ an ellipsoid of the form $Q(x) = 1$ with $Q$ a positive definite quadratic form defined over $\Q$.  Write $\Lambda$ for the lattice in $\Q^r$ consisting of those $x$ lying in $\Omega_p$ for every prime $p$.  Write $\Omega_\Q(D)$ for the (finite) intersection of $\Lambda$ with the ellipsoid $D \Omega_\R$ in $\R^r$.

Then a variable $(X_1,\ldots, X_r)$ drawn uniformly from $\Omega_\Q(D)$ satisfies 
\begin{equation}
\Pr(L_i(X_1, \ldots, X_r) = y) - \Pr(L_j(X_1, \ldots, X_r) = y) = O(D^{-2})
\label{eq:approxglobal}
\end{equation}
for all $i,j \in 1,\ldots n$ and all $y \in \Q$. 
\label{pr:approxbalanced}
\end{prop}

 Note that, here and throughout, he implicit constant in Landau notation may depend on the $L_i$.  The point is just that it is independent of $D$.

 \begin{rem}  The role of $D$ here is analogous to the role of the auxiliary line bundle $\LL_0^{\tensor d}$ in the proof of Theorem~\ref{th:ffl2g}.  We can think of the dilation by $D$ at the archimedean place as tensoring with a large power of a metrized line bundle whose degree is all concentrated at that place.
\end{rem}

\begin{proof}
Choose some $m \in \Q$ and a pair of indices $i,j$.  Write
\begin{eqnarray*}
S_i = \set{\lambda \in \Lambda: L_i(\lambda) = m, Q(\lambda) \leq D^2} \\
S_j = \set{\lambda \in \Lambda: L_j(\lambda) = m, Q(\lambda) \leq D^2}
\end{eqnarray*}

The total number of points $\lambda \in \Lambda$ with $Q(\lambda) \leq D^2$ is on order $D^r$.  So the statement to be proved is that $|S_i| - |S_j| = O(D^{r-2})$.

We have already established that $L_i$ and $L_j$ are related by an element $\gamma$ of $O_Q(\Q)$.  So
\beq
\gamma S_i = \set{\mu \in \gamma \Lambda: L_i(\gamma^{-1} \mu) = m, Q(\gamma^{-1} \mu) \leq D^2} 
=
\set{\mu \in \gamma \Lambda: L_j(\mu) = m, Q(\mu) \leq D^2}
\eeq
since $\gamma$ takes $L_i$ to $L_j$ and preserves $Q$.  

The subset of $\Lambda$ satisfying $L_i(\lambda) = m$ is the intersection of a rank $r$ lattice with an affine hyperplane, so it is either empty or a lattice of rank $r-1$ in that hyperplane.  We now use the non-archimedean ellipsoidal structure.  For each prime $p$, we have $L_i(\Omega_p) = L_j(\Omega_p)$ by hypothesis.  This implies that $L_i(\Lambda)$ and $L_j(\Lambda)$ are the same subgroup of $\Q$.  If $m$ is not in this subgroup, $S_i$ and $S_j$ are both empty and the problem is trivial, so we may assume $m$ lies in both $L_i(\Lambda)$ and $L_j(\Lambda)$.

Write $R$ for the region in $\R^r$ consisting of points $x$ with $L_j(x) = m$ and $Q(x) \leq D^2$.  Then $S_j$ and $\gamma S_i$ are each the intersection of $R$ with a rank $r$ lattice in the affine hyperplane $L_j^{-1}(m)$.  The two lattices are different; the former is $\Lambda \cap L_j^{-1}(m)$, the latter $\gamma \Lambda \cap L_j^{-1}(m)$.  However, as we have established, $\Lambda$ and $\gamma \Lambda$ have the same image in $\Q$ under $L_j$.  That means that the ratio of their covolumes in $L_j^{-1}(m)$ is the same as the ratio of their covolumes in $\R^r$.  But this latter ratio is just $|\det \gamma|$, which is $1$ because $\gamma$ is in the orthogonal group of a quadratic form.

Now let $v_j$ be the unique vector in $\R^r$ such that $L_j(x) = \langle x,v_j \rangle_Q$ for all $x \in \R^r$.  If $x$ lies on the affine hyperplane $L_j^{-1}(m)$, then $x-mv_j/Q(v_j)$ lies on the hyperplane $\ker L_j$, and moreover
\beq
Q(x-mv_j/Q(v_j)) =
Q(x) + m^2/Q(v_j) - \frac{2m}{Q(v,j)}\langle x, v_j = Q(x) - m^2/Q(v_j)
\eeq
What this means is that the region $R$ is a translate of the region $R_0$ consisting of those points $x \in \ker L_j$ such that $Q(x) \leq D^2 - m^2 / Q(v_j)$.  Translating $\gamma S_i$ and $S_j$ by $mv_j/Q(v_j)$ as well, we now find that each of $|S_i|$ and $|S_j|$ is the size of the intersection of an affine lattice with $R_0$.  

But $R_0$ is just a dilate by a factor at most $D$ of a fixed convex region of dimension $r-1$, and each of the two affine lattices are translates of a fixed lattice, namely $\Lambda \cap \ker L_j$ and $\gamma \Lambda \cap \ker L_j$.  The intersection of such a region $X$ with a translate of a fixed lattice $\Lambda_0$ is known to have size
\beq
\covol(\Lambda_0)^{-1} \vol(X) + O(D^{r-2})
\eeq
Applying this to both $\gamma S_i$ and $S_j$, and recalling that we have already shown $\Lambda \cap \ker L_j$ and $\gamma \Lambda \cap \ker L_j$ have the same covolume, we find that $|S_i| - |S_j|$ has absolute value at most $D^{r-2}$, which was the result to be proved.    
\end{proof}

Suppose now that $a_1, \ldots, a_n$ satisfy the conditions of Theorem ~\ref{th:main}.  Then Proposition~\ref{pr:localmain}, applied at each place $v$ of $K$, implies that the conditions of Proposition~\ref{pr:approxbalanced} are satisfied, i.e. the linear equation $\sum a_i x_i = 0$ admits a local ellipsoidal non-deterministic solution everywhere. It follows that the distribution on $K^{n-1}$ constructed in Proposition~\ref{pr:approxbalanced} yields variables $L_1, \ldots, L_n$ which are {\em almost} equal in distribution.

What remains is to show that an approximate non-deterministic solution to $L_1 \eqd \ldots \eqd L_n$ of this kind guaranteed by Proposition~\ref{pr:approxbalanced} can be perturbed to an exact solution. We turn to this now.

\section{Balanced and approximately balanced weightings on hypergraphs}

We can express the problem at hand in purely combinatorial terms as follows. Let $\Gamma$ be an {\em ordered $k$-uniform hypergraph} on a finite set of vertices $V$, by which we simply mean a set of ordered $k$-tuples $(v_1,\ldots, v_k)$ of elements of $V$, which we call ``edges."  An ordered $2$-uniform hypergraph is simply a directed graph (in which loops are allowed.)  Of course, we could simply replace the terminology ``ordered $k$-uniform hypergraph" with ``subset of $V^k$" but we have found the geometric point of view to be psychologically useful.

Write $E$ for the set of edges of $\Gamma$. If $e$ is an edge, we denote its vertices by $e_1, \ldots, e_k$.  A {\em balanced weighting} of $\Gamma$ is a nonzero map $w: E \ra \R_{\geq 0}$ such that, for every $v \in V$, the total weight $\sum_{e: e_i = v} w(e)$ is independent of the choice of $i$ in $1,\ldots, k$. 

The systems of nondeterministic equations studied in this paper can naturally be expressed in terms of balanced weightings.  Consider, for instance, the equation $L_1(X_1, \ldots, X_r) \eqd \ldots \eqd L_n(X_1, \ldots, X_r)$ over a global field $K$.  Let $\Gamma$ be the ordered $n$-regular hypergraph $\Gamma$ whose vertex set is $K$ and whose edges consist of those nonzero elements of $K^n$ lying in the image of $K^r$ under the map $(L_1, \ldots, L_n)$.  Then a solution of $L_1(X_1, \ldots, X_r) \eqd \ldots \eqd L_n(X_1, \ldots, X_r)$ is precisely a balanced weighting on $\Gamma$ with finite support and nonnegative integer weights, or, equivalently, such a balanced weighting on some finite subhypergraph of $\Gamma$.

Note that the issue of integer weights is not essential: the space of balanced (real-valued) weightings on $\Gamma$ is cut out by linear equations and inequalities in $\R^E$ with rational coefficients, so if this space is non-empty, it contains points with weightings in $\Q$, and because it is closed under scaling, we may dilate until the weights are in $\Z$.   

One is thus naturally led to the question of how to tell whether a finite ordered $k$-uniform hypergraph admits a balanced weighting.  When $k=2$ (i.e. when $\Gamma$ is a directed graph), a balanced weighting is a nonzero weighting of edges such that each vertex has equal weighted in-degree and weighted out-degree.  Such a weighting exists if and only if $\Gamma$ has a cycle. This is elementary, but is worth spelling out as a guide to the more involved argument to come. If $\Gamma$ has a cycle, assigning weight $1$ to the edges in the cycle and $0$ to the other edges is balanced. On the other hand, if $\Gamma$ is acyclic, there is a function $f$ from $V$ to $\R$ such that every edge in $E$ is increasing.  Let $w$ be a weighting on $\Gamma$.  Then 
\beq
\sum_v f(v) \left( \sum_{e: e_1 = v} w(e) - \sum_{e: e_2 = v} w(e) \right) = \sum_e w(e)(f(e_1) - f(e_2)).
\eeq
The right hand side is positive by the assumption on $f$ and the nonnegativity (and nonzeroness) of $w$.  But if $w$ were balanced, the left hand side would have to be zero.  So $w$ cannot be balanced.

We do not know a simple combinatorial criterion of this kind for the existence of a balanced weighting when $k > 2$. The goal of this section is to identify some conditions which are sufficient to guarantee that a balanced weighting exists, and which can be verified in the cases relevant to this paper.

We begin with some setup.  Let $\Gamma$ be an ordered $k$-uniform hypergraph with vertex set $V$ and edge set $E \subset V^k$. The vector spaces $\R^E$ and $\R^V$ are each naturally self-dual; this allows us in particular to think of any edge $e$ as an element of $\R^E$, namely the linear function that sends $e$ to $1$ and all other edges to $0$.  Let $U$ be the quotient of $(\R^V)^k$ by the diagonal copy of $\R^V$.  We now define a linear map $A: \R^E \ra U$ by
\beq
A(e) = (e_1, \ldots, e_k)
\eeq
If $w$ is a nonnegative real-valued function on $E$, then the $i$th component of $Aw$ is $\sum_{e \in E} w(e) e_i$; that is, it is the function whose value at $v$ is $\sum_{e: e_i = v} w(e)$.  So $w$ is balanced if and only if $Aw=0$. 

(We note that the map $A$ is already implicitly present in Smyth's original work~\cite[\S 4]{Smy86}.)

In the contexts we'll consider, $E$ is much larger than $V$, so the kernel of $A$ will certainly be nontrivial. But the question before us is the subtler one of whether $\ker A$ contains a nonzero vector {\em with all coordinates nonnegative.}

Suppose there is no such vector.  Then the image under $A$ of the nonnegative orthant in $\R^E$ must be contained in a half-plane in whatever plane it spans. Equivalently, there is some nonzero vector $f \in U^\vee$ such that $\langle f, Aw \rangle \geq 0$ for all nonnegative functions $w$ on $E$, and the inequality is strict for at least one $w$.  In other words, $A^\vee f$ is a nonzero vector in the nonnegative orthant of $\R^E$.  We call an $f$ satisfying this hypothesis {\em $\Gamma$-positive}. 

We may write any $f \in U^\vee$ as $f_1, \ldots, f_k$, with each $f_i$ a function on $V$, satisfying $\sum_i f_i = 0$.  (Here we have used the self-duality to identify $(\R^V)^k$ with its dual, and the quotient $U$ with the submodule $U^\vee$.)  In this representation, $f$ is $\Gamma$-positive if and only if $\sum_{i=1}^k f_i(e_i) \geq 0$ for every edge $e = (e_1, \ldots, e_k)$, and is positive for at least one $e$. 

\begin{rem}We note that, in the case $k=2$, a $\Gamma$-positive function is exactly an $(f_1, f_2)$ (with $f_2 = -f_1$ by hypothesis) where $f_1$ is nondecreasing on all edges and increasing on some edge.  When $\Gamma$ is acyclic, such a function exists; so if there is no nonzero $\Gamma$-positive function, $\Gamma$ has a cycle and thus a balanced weighting, as we have seen.
\end{rem}

We now turn to the question of proving the nonexistence of a $\Gamma$-positive function.  This is where the notion of an approximately balanced weighting comes into play.  First, some notation.  If $f = (f_1, \ldots, f_k)$ is an element of $U^\vee$, we denote by $\|f\|_1$ the norm $\sum_{i,v} |f_i(v)|$, and by $\|f\|_\infty$ the norm $\sup_{i,v} |f_i(v)|$.

Now let $c>0$ be a constant and suppose $w$ is a function on $E$ taking values between $c$ and $1$.  Let $f$ be an element of $U^\vee$.  Then
\beq
\langle w, A^\vee f \rangle 
= \langle Aw, f \rangle 
\leq \|Aw\|_\infty \|f\|_1.
\eeq
On the other hand, if $f$ is $\Gamma$-positive,
\beq
\langle w, A^\vee f \rangle \geq c \|A^\vee f\|_1.
\eeq
Putting these together, we find that
\begin{equation}
\|A^\vee f\|_1 \leq c^{-1} \|Aw\|_\infty \|f\|_1
\label{eq:atfinequality}
\end{equation}
We will not give a formal definition for ``approximately balanced" except to say it means $\|Aw\|_\infty$ is small; in other words, the imbalance of $w$ is small at {\em every} vertex.  What the argument above shows is that, if $\Gamma$ admits an approximately balanced weighting, then every $\Gamma$-positive $f$ is unexpectedly small in $L^1$ norm when $A^\vee$ is applied.  (We note in passing that, if $w$ is balanced on the nose, the above argument shows that there is no $\Gamma$-positive $f$.)

\begin{rem}
   It is worth noting that a more general combinatorial approach could potentially take the place of the arguments in the following section, which rely on special features of hypergraphs arising from nondeterministic linear systems.  The weightings we construct are quite special; when $K=\Q$, the hypergraphs in question have roughly $N^{k-1}$ edges and $N$ vertices for some parameter $N$, the weighting $w$ is the constant function $\mathbf{1}$, and $\|A\mathbf{1}\|_\infty$ is bounded above by a constant independent of $N$.  In general, one can ask:  If $\Gamma,E,V,A$ are defined as above, and if $\Gamma$ admits no balanced weighting, how small can $\|A\mathbf{1}\|_\infty$ be in terms of $|E|$ and $|V|$?  We do not even know the answer for $k=2$, in which case the question is: for a directed acyclic graph with $|E|$ edges on $|V|$ vertices, how small can
   \beq
   \max_{v \in V} (\mbox{in-degree}(v) - \mbox{out-degree}(v))
   \eeq
   be?
\end{rem}

\section{Proof of Theorem~\ref{th:main}}

From now on, we let $K=\Q$, and return to our consideration of a nondeterministic system of linear equations $L_1(X_1, \ldots, X_{n-1}) \eqd \ldots \eqd L_n(X_1, \ldots, X_{n-1})$ over $\Q$.  We suppose that the $L_i$ satisfy only a single linear relation, which we denote $\sum a_i L_i = 0$, and we assume the $a_i$ satisfy the local conditions of Theorem~\ref{th:main}.  Then Proposition~\ref{pr:localmain} and  Proposition~\ref{pr:approxbalanced} tells us that for each real $D > 0$ there exists a finite subset of $\Q^{n-1}$, there denoted $\Omega_\Q(D)$, which provides an approximate solution to the nondeterministic linear system.  Now and for the rest of this section, all constants, including implied constants in  Landau notation $f = o(g), f = O(g), f = \Theta(g)$, are allowed to depend on $n$, and the linear forms $L_1, \ldots, L_n$; what's important is that they don't vary with $D$.

We recall the key properties of $\Omega_\Q(D)$ established in section~\ref{s:approxl2g}:

\begin{itemize}
    \item $\Omega_\Q(D)$ is the intersection of a fixed lattice $\Lambda$ in $\R^{n-1}$ with the dilate $D \Omega_\R$ of a fixed ellipsoid $\Omega_\R \subset \R^{n-1}$.  In particular, $|\Omega_\Q(D)| = c D^{n-1} + o(D^{n-1})$. (Better error terms are possible, but aren't necessary for us.)  
    \item The image $L_i(\Lambda)$ is the same lattice in $\Q$ for each $i$.  Scaling if necessary, we can and do take this lattice to be $\Z$.  The real number $\mu = \max L_i(\Omega \R)$ is also independent of $i$.
    \item For any $m \in \Z$, the difference between the number of points $x$ in $\Omega_\Q(D)$ with $L_i(x) = m$ and the number of points with $L_j(x) = m$ is $O(D^{n-3})$.  (This is the output of Proposition~\ref{pr:approxbalanced}.) 
\end{itemize}

The following elementary lemma will be useful in several places; it shows that lattice points sufficiently near the edge of $\Omega_\Q(D)$ can take extreme values at only one $L_i$.  Note that $|L_i(x)| \leq \mu D$ for any $x \ \in \Omega_\Q(D)$.

\begin{lem}  There is a constant $\delta > 0$ such that, for any $x$ in $\Omega_\Q(D)$, there is at most one $i$ such that $|L_i(x)| > (1-\delta) \mu D$.
\label{le:deltaedge}
\end{lem}

\begin{proof}
The ellipsoid $\Omega_\R$ is tangent to the hyperplane $L_i(x) = \mu$ and the tangent point is the unique $x_i \in \Omega_\R$ such that $L_i(x_i) = \mu$.  This implies that $L_j(x_i)$ cannot also be $\mu$, since the tangent hyperplane to $\Omega_\R$ at $x_i$ is parallel to $L_i$, not $L_j$.  We conclude that $L_j(x_i)$ is strictly less than $\mu$; in particular, there is some $\delta_j$ such that for all $x \in \Omega_\R$ with $L_i(x) > (1-\delta_j) \mu$, we have $L_j(x) < (1-\delta_j) \mu$.  Take $\delta$ to be the minimum over all $i \neq j$ of $\delta_j$.  Then an $x$ with $L_i(x) > (1-\delta)\mu$ must have $L_j(x) < (1-\delta)\mu$ for all $j \neq i$.  Because any $L_j$ is related to $L_i$ by a symmetry of $\Omega_\R$, this $\delta$ fulfills the requirements of the theorem. 
\end{proof}

As promised, we are going to show that it's possible to perturb the uniform distribution on $\Omega_\Q(D)$ in order to arrive at an exact solution to $L_1 \eqd \ldots \eqd L_n$.  One problem that immediately faces us is that there may be integers $m \in \Z$ which lie in $L_i(\Omega_\Q(D)$ but not in $L_j(\Omega_\Q(D)$ for some $i,j$. In such a situation, we cannot assign any positive probability to a point $(x_1, \ldots, x_{n-1})$ with $x_i = m$, since that would make $\Pr(X_i = m)$ positive, while $\Pr(X_j = m)$ is necessarily zero.  Constraints like this will make things more complicated later, so we repair it now.

\begin{prop} There are arbitrarily large values of $D$ such that $L_i(\Omega_\Q(D))$ is the same subset of $\Z$ for all $i$.
\label{pr:sameimage}
\end{prop}

\begin{proof}  Let $x_j$ be the point in the ellipsoid $\Omega_\R$ where $\Omega_\R$ is tangent to $L_i(x) = \mu$.  Then $x_j$ has rational coordinates.  Choose $D$ to be divisible enough that $L_i(Dx_i) = D \mu$ is an integer.  

Let $m$ be an integer in $[-D\mu,D\mu]$.  We recall the notation from the proof of Proposition~\ref{pr:approxbalanced}, in which $v_i$ is the vector in $\Q^r$ such that $L_i(x) = \langle x,v_i \rangle_Q$ for all $x \in \Q^r$.  We showed in that proof that the subset of $x \in \Omega_{\Q}(D)$ with $L_i(x) = m$ is the translate by $m v_i / Q(v_i)$ of the subset of points $y$ in the real vector space $\ker L_i$ with $Q(y) \leq D^2 - m^2 / Q(v_j)$ and $y$ lying in a certain translate of the lattice $\Lambda_i = \Lambda \cap \ker L_i$.

First of all, we note that this set consists of a single point when $m = D \mu$, so $D^2 - (D \mu)^2 / Q(v_j) = 0$; in other words,  $Q(v_j) = \mu^2$.

Now the minimum value that $D^2 - m^2 / Q(v_j)$ can take for $m \in [-D\mu,D\mu]$ is attained when $m = D \mu - 1$, in which case we have
\beq
D^2 - m^2 / Q(v_j) = D^2 - (D^2 - 2 D \mu^{-1} + \mu^{-2}) = 2 D \mu^{-1} - \mu^{-2}.
\eeq
On the other hand, there is some $C$ large enough such that the subset of $\ker L_i$ satisfying $Q(x) < C$ contains a point of any translate of the lattice $\Lambda_i$.  Choosing $D$ large enough that $2 D \mu^{-1} - \mu^{-2} > C$, we have shown that the image of $\Omega_\Q(D)$ under $L_i$ is the set of all integers in the interval $[-D\mu, D\mu]$. 
 The same argument applies to each $i \in 1,\ldots, n$, so we are done.

\end{proof}

From this point onward, we always require $D$ to be an integer satisfying the conditions of Proposition~\ref{pr:sameimage}; we denote the interval $[-D\mu, D\mu]$ in $\Z$ by $V_D$ and we denote $\Omega_\Q(D)$ by $E_D$, since we are thinking of these as the vertices and edges of an ordered $n$-uniform hypergraph $\Gamma$.  The theorem to be proved is that $\Gamma$ admits a balanced weighting.

We maintain the notation of the previous section:  $U_D^\vee$ is the subspace of $(\R^{V_D})^n$ consisting of $n$-tuples of functions summing to $0$, and a function $f \in U_D^\vee$ is $\Gamma$-positive if $\sum_{i=1}^n f_i(e_i) \geq 0$ for all edges $e \in E_D$, and positive for at least one edge.  The existence of a balanced weighting follows if we can prove there are no $\Gamma$-positive functions in $U_D^\vee$. (The reader may find it helpful to notice that determining the existence of a balanced weighting for $\Gamma$ can be viewed as a linear programming feasibility problem; the existence of a $\Gamma$-positive function in $U_D^\vee$ is then equivalent to the dual program being unbounded, which is equivalent to the primal program being infeasible.)

The result of Proposition~\ref{pr:approxbalanced} tells us that $\|A\mathbf{1}\|_\infty = O(D^{n-3})$, where $\mathbf{1}$ denotes the constant weighting $1$ on $E_D$.  From this fact, it follows by \eqref{eq:atfinequality} that, if $f$ is a $\Gamma$-positive function,
\begin{equation}
\|A^\vee f\|_1 = O(D^{n-3}\|f\|_1).
\label{eq:ellipseatfinequality}
\end{equation}
Since each value $A^\vee f(e)$ is a sum $\sum_{i=1}^n f_i(e_i)$ of $n$ values, and $|E_D|$ is of order $D^{n-1}$ while $|V_D|$ is of order $D$, one might expect the $L^1$ norm of $A^\vee f$ to be about $D^{n-2}$ times that of $f$; thus, \eqref{eq:ellipseatfinequality} can be thought of as asserting that a $\Gamma$-positive function $f$ undergoes a substantial amount of cancellation under $A^\vee$; or, in other words, that $f$ is unexpectedly large in $L^1$ relative to $A^\vee f$.

It would be convenient were this to be impossible, but that's too much to ask.  One problem is the presence of elements of vertices with very small degree in $\Gamma$; these occur near the edge of the interval $V_D$.  For instance, we have specified that the maximal element of $V_D$ occurs as $e_i$ for {\em just one} edge $e \in E_D$, which means that for a function $f$ supported at that element, $\|A^\vee f\|$ will be $O(\|f\|_1)$, much smaller than $D^{n-2} \|f\|_1$.

And that's not the only problem.  There are nonzero functions in the kernel of $A^\vee$, such as constant functions (and in some cases there are others as well) and when $g$ is such a function, $A^\vee g$ is of course not large compared to $g$. And if $g \in \ker A^\vee$, then adding a large multiple of $g$ to any $f$ will produce a function which shrinks substantially under $A^\vee$.

The first and most involved part of our argument will show that, for the hypergraphs considered here, these are the only two problems; we can modify any $f$ by an element of $\ker A^\vee$ to get a function whose $L^1$ norm, away from some region near the edge of $V_D$, is no larger than we expect it to be.  This part of the argument does not require $f$ to be $\Gamma$-positive.

\medskip

Let $\epsilon > 0$ be a small constant (in particular, smaller than the $\delta$ in Proposition~\ref{le:deltaedge}) and write $V_{(1-\epsilon)D}$ for the set of integers between $-D(1-\epsilon)\mu$ and $D(1-\epsilon)\mu$.  Having chosen $\epsilon$, we keep it fixed throughout the proof, and all implied constants in Landau notation are allowed to depend on $\epsilon$.

\begin{prop} There exists a constant $C$ (depending on the $L_i$, but not on $D$) with the following property:  every coset in $U_D^\vee / \ker A^\vee$ contains an element $f$ such that
\beq
\|f|{V_{(1-\epsilon)D}}\|_1 \leq C D^{2-n} \|A^\vee f\|_1.
\eeq
\label{pr:smalloninterior}
\end{prop}

\begin{proof}
Let $y$ and $y'$ be nonzero elements of $\Lambda$ such that, for each $j > 1$, either $L_j(y)$ or $L_j(y')$ is $0$.  (For instance, we might take $y_1$ to be in the kernel of $L_3, \ldots, L_n$ and $y_2$ in the kernel of $L_2, L_4, \ldots, L_n$.)  We emphasize that $y$ and $y'$ are independent of $D$.  Write $a,a'$ for $L_1(y),L_1(y')$.  Then, for any $x$ in $E_D$ such that $x,x+y,x+y'$, and $x+y+y'$ all lie in $E_D$, we have
\begin{equation}
\begin{array}{rcl}
    A^\vee f (x) - A^\vee f (x+y) - A^\vee f (x+y') + A^\vee f (x+y+y') & = & \\
      f_1(L_1(x)) - f_1(L_1(x) + a)  - f_1(L_1(x) + a') + f_1(L_1(x) + a + a'). \\
\end{array}
\label{eq:parallelogram}
\end{equation}

The identity \eqref{eq:parallelogram} is our key tool for controlling $f_1$ (and of course any other $f_i$) given some control over the values of $A^\vee f$.  

\begin{rem}
    This is one place where we really use the condition that $r=n-1$.  When $r$ is much smaller than $n$, there won't be a pair $y,y'$ such that all $L_i$ other than $L_1$ vanish on either $y$ or $y'$; it will be necessary to use larger sets of $y$, which in turn means replacing the parallelograms in \eqref{eq:parallelogram} with higher-dimensional parallelipipeds, and later in the argument replacing discrete second derivatives with discrete higher-order derivatives.  We find it plausible, but certainly not obvious, that this can be made to work.
\end{rem}

For any integer $r$ and any function $F$ on $\Z$, we write $\Delta_r F$ for the discrete derivative function given by $\Delta_r F(m) = F(m+r) - F(m)$.  Then the right-hand side of \eqref{eq:parallelogram} can be more compactly written as $\Delta_a \Delta_a' f_1(L_1(x))$.

Fix another small constant $\eta > 0$, and let $k,k'$ be integers which are less than $\eta D$.  Let $P_{k,k',D}$ be the set of all parallelograms $\mathbf{p}_x = x,x+ky,x+k'y',x+ky+k'y'$ which are contained in $E_D$.  We now consider the sum
\beq
S_{k,k',D} = \sum_{\mathbf{p}_x \in P_{k,k',D}} A^\vee f (x) - A^\vee f (x+ky) - A^\vee f (x+k'y') + A^\vee f (x+ky+k'y').
\eeq
On the one hand, if $e$ is an element of $E_D$, the value $A^\vee f (e)$ occurs at most four times in $S_{k,k',D}$.  So $S_{k,k',D} \leq 4\|A^\vee f\|_1$.

On the other hand, applying \eqref{eq:parallelogram} tells us that
\beq
S_{k,k',D} = \sum_{\mathbf{p}_x \in P_{k,k',D}} \Delta_{ka} \Delta_{k'a'} f_1(L_1(x)).
\eeq
In order to know how many times a given value $\Delta_{ka} \Delta_{k'a'} f(m)$ appears in this latter sum, we need to know how many parallelograms in $P_{k,k',D}$ are based at an $x \in E_D$ with $L_1(x) = m$.  In other words, we are asking for the number of $x \in E_D$ such that $L_1(x) = m$ and $x,x+ky,x+k'y',x+ky + k'y'$ all lie in $E_D$. 
 When $m$ is very close to the edge of $V_D$, this number can certainly be small; for example, when $m$ is equal to its extremal value $D\mu$ there are no such elements of $E_D$. We claim, however, that when $m \in V_{(1-\epsilon)D}$, this number is bounded below by $c_{\eta,\epsilon} D^{n-2}$, where $c_{\eta,\epsilon} > 0$ so long as $\eta$ is small enough relative to $\epsilon$.  To see this, consider the region $R$ in $\Omega_\R$ consisting of all $x$ such that $L_1(x) = m/D$ and $x + \eta y, x + \eta y',$ and $x + \eta y + \eta y'$ all lie inside $\Omega_\R$.  As long as $\eta$ is small enough relative to $\epsilon$, this region is a positive-volume convex region inside the ellipsoidal slice $L_1(x) = m/D$.  What's more, a point $x$ in $E_D = \Lambda \cap D \Omega_\R$ is the base of a parallelogram in $P_{k,k',D}$ if and only if it lies in $\Lambda \cap DR$. This finite set is the intersection in an $(n-2)$-dimensional hyperplane of the $D$-dilate of a fixed positive-volume region with a translate of a fixed lattice (namely, the lattice $\Lambda_i = \Lambda \cap \ker L_i$.) The cardinality of the intersection thus grows with order $D^{n-2}$.  This shows that
\beq
S_{k,k',D} \geq \sum_{m \in V_{(1-\epsilon)D}} c_{\epsilon,\eta} D^{n-2} |\Delta_{ka} \Delta_{k'a'} f_1(m)|.
\eeq

Combining the upper and lower bounds for $S_{k,k',D}$ obtained above, and absorbing the dependence on $\eta$ into that on $\epsilon$, we have now established that, as long as $\eta$ is small enough relative to $\epsilon$,
\begin{equation}
\sum_{m \in V_{(1-\epsilon)D}} |\Delta_{ka} \Delta_{k'a'} f_1| = O(D^{2-n} \|A^\vee f\|_1)
\label{eq:DeltaDeltabound}
\end{equation}
for any $k,k' \leq \eta D$.

As one might expect, the bound above on the size of double discrete derivatives will allow us to show that $f_1$ is well-approximated by a linear function, at least on residue classes of bounded modulus.  We will need the following three elementary lemmas.  It seems very plausible that the third lemma in the chain, which is the one we need, can be proven more directly, and we invite the reader to tell us how to do so.

\begin{lem} Let $x_1, \ldots, x_M$ be a set of $M$ real numbers and let $\EE x$ be their mean.  Then
\begin{equation}
\sum_{i=1}^M |x_i - \EE x| \leq 
(1/M) \sum_{i,j} |x_i - x_j|.
\label{eq:smalldifferences}
\end{equation}
\label{le:smalldifferences}
\end{lem}

\begin{proof}  This follows from Jensen's inequality.  If $X$ and $Y$ are random variables drawn uniformly and independently from $x_1, \ldots, x_M$, then the left-hand side of \eqref{eq:smalldifferences} is $M \EE_X |\EE_Y(X-Y)|$, while the right-hand side is $M \EE_X \EE_Y (|X-Y|)$.
\end{proof}

\begin{lem} Let $\eta, C$ be positive constants.  Let $F$ be a function on $1, \ldots M$ such that $\|\Delta_k F\|_1 \leq C$ for all $1 \leq k \leq \eta M$.  Then
\beq
\|F-\EE F\|_1 = O_\eta(C).
\label{le:shortintervals}
\eeq
\end{lem}

\begin{proof} 
(First, note that $\Delta_k F$ is a function on $1, \ldots, M-k$, not $1,\ldots M$; this will not matter in the proof.

Consider an interval $I = [x,x+\eta M]$.  We have
\beq
\sum_{x,y \in I} |F(x) - F(y)| \leq \sum_{k=1}^{\eta M} \|\Delta_k F\|_1 \leq \eta M C.
\eeq

Write $A_I$ for the average of $F$ over $I$. By Lemma~\ref{le:smalldifferences} we have
\beq
\sum_{x \in I} |F(x) - A_I| \leq M^{-1} \sum_{x,y \in I} |F(x) - F(y)|. 
\eeq

Putting these together, we get
\beq
\sum_{x \in I} |F(x) - A_I| \leq \eta C
\eeq
which is to say that $F(x)$ is approximately constant on this interval.  If $I_1$ and $I_2$ are intervals of size $\eta M$ whose overlap $I_{12}$ has size $(1/2)\eta M$, we see that
\beq
\sum_{x \in I_{12}} |F(x) - A_{I_1}| \leq \sum_{x \in I_1} |F(x) - A_{I_1}|
\leq \eta C
\eeq
and similarly for $|F(x) - A_{I_2}|$; subtracting these two functions, we find that
\beq
\sum_{x \in I_{12}} |A_{I_1} - A_{I_2}| \leq 2 \eta C
\eeq
which is to say the two means differ by at most $4 C / M$.  But since we can get from any length-$\eta M$ interval to any other in a bounded number of steps of this kind, all the $A_I$ lie within $O_\eta(C/M)$ of each other, and thus the overall mean $\EE F$ lies within $O_\eta(C/M)$ of each $A_I$. Combined with our estimate for the $L^1$ difference between $F$ and $A_I$ on $I$, this yields the desired result.

\end{proof}

The following lemma is analogous to Lemma~\ref{le:shortintervals}, but with double discrete derivatives instead of single ones.  Just as Lemma~\ref{le:shortintervals} says that a function with many $L^1$-small discrete derivatives must be $L^1$-close to a constant, this one says that a function with many $L^1$-small double discrete derivatives must be $L^1$-close to a linear function.  One imagines that the analogous statement for derivatives of any degree is true as well, and perhaps is even known, but we were not able to find a statement to this effect in the literature.

\begin{lem}
    Let $\eta,C$ be positive constants.  Let $F$ be a function on $1,\ldots,M$ and suppose that $\|\Delta_{k'} \Delta_k F \|_1 \leq C$ for all $1 \leq k,k' \leq \eta M$.  Then there is a linear function $L(x) = mx+b$ such that $\|F - L\|_1 = O(C)$.
    \label{le:linapprox}
\end{lem}

\begin{proof}
It follows from Lemma~\ref{le:shortintervals} that $\|\Delta_{k} F - \EE \Delta_{k} F\|_1 = O_\eta(C)$ for every $k \leq \eta M$.  Choose a $k$ which is a small multiple of $\eta M$; say, between $(1/10) \eta M$ and $(1/20) \eta M$.  Write $\gamma$ for $\EE \Delta_{k} F$.  Note first that the function $F_\gamma$ sending $x$ to $F(x) - \gamma x$ satisfies $\Delta_{k'} \Delta_k F_\gamma = \Delta_{k'} \Delta_{k} F$ for all $k,k'$, and $F_\gamma$ is $L^1$-close to a line  if $F$ is. So without loss of generality we may assume that $\gamma = 0$, whence $\| \Delta_k F \|_1 = O_\eta(C)$.

Now consider the function $\tilde{F}$ defined by $\tilde{F}(x) = F(\bar{x})$ where $\bar{x}$ is the unique integer in $1,\ldots,k$ congruent to $x$ mod $k$.  The function $\tilde{F}$ is periodic; it is pulled back from a function $\bar{F}$ under the reduction map from $1,\ldots, M$ to $\Z/kZ$.  We claim that $\|\tilde{F} - F\|_1 = O_\eta(C)$.  To see this, note that the sum of $F(x) - \tilde{F}(x) = F(x) - F(\bar{x})$ over all $x$ can be expressed as a sum $\Delta_k F(\bar{x}) + \Delta_k F(\bar{x} + k) + \ldots + \Delta_k F(x - k)$, a sum of $O_\eta(1)$ values of $\Delta_k$ at values congruent to $x$ mod $k$.  So the sum of $F(x) - F(\bar{x})$ is a sum of values of $\Delta_k$ which includes any given $\Delta_k(y)$ at most $O_\eta(1)$ times; this bounds $\|\tilde{F} - F\|_1$ by a constant multiple of $\|\Delta_k\|_1$, as claimed.

It thus suffices to show the existence of a linear function $L$ such that $\|\tilde{F} - L\|_1 = O_\eta(C)$. Let $\bar{F}$ be the function on $\Z/k\Z$ from which $\tilde{F}$ is pulled back.  Let $\bar{y}$ be an element of $\Z/k\Z$ and $y$ a lift of $\bar{y}$ lying in $1, \ldots, \eta M$.  Then
\beq
\sum_{\bar{x} \in \Z/k\Z} |\bar{F}(\bar{x}+\bar{y}) - \bar{F}(\bar{x})|
= \sum_{x=1}^k |\tilde{F}(x+y) - \tilde{F}(x)| = \sum_{x=1}^k |\Delta_y \tilde{F}(x)|.
\eeq

We now argue that this last sum is $O_\eta(C)$.
First of all, we know from another application of Lemma~\ref{le:shortintervals} that $\|\Delta_y F\ - \EE \Delta_y F\|_1 = O_\eta(C)$.  So the shorter sum $\sum_{x=1}^k |\Delta_y F(x) - \EE \Delta_y F|$ is also $O_\eta(C)$.  And $\sum_{x=1}^k |\Delta_y F(x) - \Delta_y \tilde{F}(x)|$ is smaller than  
\beq
\|\Delta_y \tilde{F} - \Delta_y F\|_1 =
\|\Delta_y(\tilde{F} - F)\|_1 =
O_\eta(\|\tilde{F}-F\|_1) = O_\eta(C).
\eeq
Putting these together, we find that
\beq
\sum_{x=1}^k |\Delta_y \tilde{F}(x) - \EE \Delta_y F| = O_\eta(C).
\eeq
In other words, the total absolute deviation of $\Delta_y \tilde{F}$ from the constant $\EE \Delta_y F$ on $1, \ldots, k$ is $O_\eta(C)$.

In fact, we can show that the constant $\EE \Delta_y F$ is small enough to be safely excluded, by another application of Jensen's inequality as in Lemma~\ref{le:smalldifferences}, as we now explain. Let $X$ be a random variable drawn uniformly from the $k$ values of $\Delta_y \tilde{F}(x)$ as $x$ ranges over $1, \ldots, k$.  The mean of $X$ is zero, since $\tilde{F}$ is periodic in the range $1,\ldots, k+y$.  Jensen's inequality then gives
\beq
\EE|X - \EE \Delta_y F| \geq |\EE(X - \EE \Delta_y F)| = |\EE \Delta_y F|
\eeq
so $|\EE \Delta_y F| = O_\eta(C/M)$.  Moreover,
\beq
\EE|X| \leq \EE|X - \EE \Delta_y F| + |\EE \Delta_y F| = O_\eta(C/M).
\eeq

We conclude that
\begin{equation}
\sum_{x=1}^k |\Delta_y \tilde{F}(x)| = O_\eta(C).
\label{eq:deltayl1small}
\end{equation}

Summing \eqref{eq:deltayl1small} over all $y \in \Z/k\Z$, we find that the sum of {\em all} the absolute differences between values of $\bar{F}$ is $O_\eta(CM)$.  Lemma~\ref{le:smalldifferences} now shows that $|\bar{F} - \EE\bar{F}\|_1 = O_\eta(C)$, and since $\tilde{F}$ consists of $O_\eta(1)$ copies of $\bar{F}$, we have that $\|\tilde{F} - \EE\bar{F}\|_1 = O_\eta(C)$.  This completes the proof.
\end{proof}

(Note that, having left the lemmas and returned to the main body of proof, we once again treat $\eta$ as a fixed constant and implicit constants in Landau notation are allowed to depend on it without explicit notational permission.)

We have shown in \eqref{eq:DeltaDeltabound} that $f_1$ has many discrete second derivatives which are small in $L^1$ norm after restriction to $V_{(1-\epsilon)D}$.  If $\Delta_k \Delta_k' f_1$ were small for {\em all} small enough $k,k'$, we could apply Lemma~\ref{le:linapprox} to show that $f_1$ was well-approximated by a linear function.  But we do not have all those second derivatives, though; \eqref{eq:DeltaDeltabound} only gives us access to $\Delta_{ka} \Delta_{k'a'}$ with $k, k' \in \Z$.  To deal with this congruence condition, we are going to let $P$ be a large integer (but still, crucially, bounded independently of $D$) which is a multiple of both $a$ and $a'$, and which indeed is also a multiple of the corresponding $a,a'$ that arise when we treat $f_2, \ldots, f_n$ in place of $f_1$.  We restrict to a congruence class $d \pmod P$ (and we continue to restrict to the subinterval $V_{(1-\epsilon)D}$).  Having done so, we may define a function $F_i(x) = f_i(\frac{x-d}{P})$ on this subset of $V_D$; then \eqref{eq:DeltaDeltabound} tells us that
\begin{equation}
\| \Delta_k \Delta_{k'} F_i \|_1 = O(D^{2-n} \|A^\vee f\|_1)
\label{eq:secondderivapprox}
\end{equation}
for all $k,k'$ in a range of size proportional to $D$.  Now Lemma~\ref{le:linapprox} tells us that, for all $i$, there is a linear function $\ell_i$ (calling it $L_i$ would conflict with earlier notation) such that
\begin{equation}
\|F_i - \ell_i\|_1 = O(D^{2-n} \|A^\vee f\|_1).
\label{eq:linapprox}
\end{equation}

We will now show that these linear functions can be taken to be constants.  For this, we will need to think about the relationship between the different $f_i$.  First of all, recall that the linear forms $L_i$ obey a unique linear relation, which we denote
\beq
a_1 L_1 + \ldots a_n L_n = 0.
\eeq

Now let $y_{ij}$ be a vector in $\Lambda$ such that $L_\iota(y_{ij}) = 0$ for all $\iota$ other than $i,j$.  The $L_i(y_{ij})$ and $L_j(y_{ij})$ must be in the ratio $a_j:a_i$; write them as $r a_j$ and $r a_i$. 

As usual, let $k'$ be a positive integer smaller than $\eta D$. Then, for any $x \in E_D$ such that $x+k'y_{ij}$ is also in $E_D$, we have
\begin{equation}
A^\vee f (x+k'y_{ij}) - A^\vee f (x)
=
\Delta_{k' r a_j} f_i(L_i(x)) - \Delta_{k'r a_i} f_j(L_j(x)).
\label{eq:segment}    
\end{equation}

Now choose some $k'$ with $k'r = \kappa P$ a multiple of $P$.  Let $S$ be the set of $x$ in $E_D$ such that $L_i(x)$ lies in the domain of $F_i$ and $L_j(x)$ lies in the domain of $F_j$; then $|S|$ is on order $D^{n-1}$.  Now sum \eqref{eq:segment} over all $x \in S$. The left-hand side of this sum is a sum of values of $A^\vee f$ in which each value appears at most twice; so its size is $O(\|A^\vee f\|_1)$.  The right-hand side of the sum, on the other hand, is a sum of quantities of the form
\beq
\Delta_{\kappa a_j} F_i(m_i) - \Delta_{\kappa a_i} F_j(m_j)
\eeq
in which each $m_i$ (and likewise each $m_j$) appears $O(D^{n-2})$ times.

We now use the fact that $F_i$ and $F_j$ are each well-approximated by linear functions $\ell_i(m) = \gamma_i(m) + b_i$ and $\ell_j(m) = \gamma_j(m) + b_j$.  In particular, \eqref{eq:linapprox} implies that 
\beq
\| \Delta_{\kappa a_j} F_i - \Delta_{\kappa a_j} \ell_i \|_1 = \| \Delta_{\kappa a_j} F_i - \kappa \gamma_i a_j \|_1 = O(D^{2-n} \|A^\vee f\|_1)
\eeq
Similarly, $\Delta_{\kappa a_i} F_j$ is $L^1$ close to the constant $\kappa \gamma_j a_i$.  It follows that
\beq
\sum_{x \in S} \Delta_{k' r a_j} f_i(L_i(x))
= \kappa \gamma_j a_i |S| + O(\|A^\vee f\|_1)
\eeq
since each value of $L_i(x)$ appears $O(D^{n-2})$ times as $x$ ranges over $S$, as mentioned above.  Thus

\beq
\sum_{x \in S} \Delta_{k' r a_j} f_i(L_i(x)) - \Delta_{k'r a_i} f_j(L_j(x))
= \kappa (\gamma_i a_j - \gamma_j a_i) |S| + O(\|A^\vee f\|_1).
\eeq

But we have already observed that the sum of the right-hand side of \eqref{eq:segment} is $O(\|A^\vee f\|_1)$.  We thus have
\beq
\kappa (\gamma_i a_j - \gamma_j a_i) |S| = O(\|A^\vee f\|_1).
\eeq

We can take $\kappa$ to be of order $D$, and we have already observed $|S|$ is of order $D^{n-1}$.  This yields
\begin{equation}
\gamma_i a_j - \gamma_j a_i = O(D^{-n} \|A^\vee f\|_1)
\label{eq:gammaproportions}
\end{equation}

In other words, the ratio of $\gamma_i$ to $\gamma_j$ is the same as that of $a_i$ to $a_j$, up to some error whose size we can control.  It follows that, for each $i$,
\begin{equation}
a_i \sum_{j=1}^n \gamma_j = 
\gamma_i(a_1 + a_2 + \ldots + a_n)  + O(D^{-n} \|A^\vee f\|_1).
\label{eq:sumai}
\end{equation}

On the other hand, $\sum_{i=1}^n \gamma_i$ can also be bounded above, as we now explain.  Recall that $\sum_{i=1}^n F_i = 0$.  It then follows from \eqref{eq:linapprox} that 
\beq
\|\sum_{i=1}^n \ell_i\|_1 = O(D^{2-n} \|A^\vee f\|_1).
\eeq
On the other hand, $\sum_{i=1}^n \ell_i$ is a linear function given by
\beq
\sum_{i=1}^n \ell_i(x) = (\sum_{i=1}^n \gamma_i)x + \sum_{i=1}^n b_i
\eeq
and it is easy to see that the $L^1$ norm of this linear function is bounded below by a constant multiple of $(\sum_{i=1}^n \gamma_i ) D^2$.  Putting this together with the upper bound above for $\|\sum_{i=1}^n \ell_i\|_1$ yields
\beq
\sum_{i=1}^n \gamma_i = O(D^{-n} \|A^\vee f\|_1).
\eeq
Combining this with \eqref{eq:sumai} shows that, for each $i$, 
\beq
\gamma_i (a_1 + a_2 + \ldots + a_n) = O(D^{-n} \|A^\vee f\|_1).
\eeq

As long as $a_1 + \ldots + a_n \neq 0$, this shows that $\gamma_i = O(D^{-n} \|A^\vee f\|_1)$.  On the other hand, if $a_1 + \ldots + a_n = 0$, the element $g$ of $U$ given by $g_i(x) = a_i x$ lies in $\ker A^\vee$, since for every point $x$ of $E_D$, we have $A^\vee g(x) = \sum_{i=1}^n g_i(L_i(x)) = \sum_{i=1}^n a_i L_i(x) = 0$.  
 So we have the freedom to subtract a multiple of $g$ from $f$.  In particular, we may replace $f$ with $f-(\gamma_1 / a_1)g$; in so doing, we replace each $\gamma_i$ with $\gamma_i - (\gamma_1/a_1) a_i$.  By \eqref{eq:gammaproportions}, this quantity is $O(D^{-n} \|A^\vee f\|_1)$.  Thus we can in this case too ensure that $\gamma_i = O(D^{-n} \|A^\vee f\|_1)$ for all $i$.

We have thus bounded the slope of the linear function $\ell_i(m) = \gamma_i m + b_i$, and in particular this yields a bound
\beq
\| \ell_i - b_i \|_1 = O(D^{2-n} \|A^\vee f\|_1).
\eeq
Since \eqref{eq:linapprox} gives a bound of the same size for the $L^1$ distance between $F_i$ and $\ell_i$, we conclude that
\begin{equation}
\| F_i - b_i \|_1 = O(D^{2-n} \|A^\vee f\|_1).
\label{eq:approxconstmodP}
\end{equation}

To sum up:  we have now shown that $F$ is well-approximated by a constant function.  And since $F$ is the restriction of $f$ to a congruence class of bounded modulus $P$, and since we can apply this argument to each congruence class of modulus $P$, we have shown that $f$ is well-approximated (on the interval $V_{(1-\epsilon)D}$, as always) by a function on $V_D$ that is {\em periodic} with period $P$.  We write $\tilde{f}$ for this periodic function.  Then $A^\vee \tilde{f}$ is also constant on congruence classes mod $P$. 

For any function $F$ on $V_D$, write $\|F\|_1^{int}$ (for ``interior") for the sum of $|F(m)|$ as $m$ ranges over the subinterval $V_{(1-\epsilon)D}$.  Similarly, for any function $G$ on $E_D$, write $\|G\|_1^{int}$ for the sum of $|G(x)|$ over those $x$ such that $L_i(x) \in V_{(1-\epsilon)D}$ for all $i$.  (We recall that $\epsilon$ has been chosen to be small enough that such $x$ make up a positive proportion of the edges in $E_D$.)

It then follows from \eqref{eq:approxconstmodP} that
\begin{equation}
\|f-\tilde{f}\|_1^{int} = O(D^{2-n}\|A^\vee f\|_1)
\label{eq:ftildefint}
\end{equation}

We now turn to bounds for the periodic function $\tilde{f}$.  For this purpose, we can set up our whole problem in the $\Z/P\Z$ context.  Write $\bar{U}$ for the space of $n$-tuples of functions $\bar{f}_1, \ldots, \bar{f}_n$ which sum to $0$.  Then there is a linear transformation $\bar{A}^\vee$ from $\bar{U}$ to the space of functions on $(\Z/P\Z)^{n-1}$ defined by
\beq
\bar{A}^\vee f (\bar{x}) = \sum \bar{f}_i(L_i(\bar{x})).
\eeq
Write $\bar{U}_0$ for the orthogonal complement of $\ker \bar{A}^\vee$ in $\bar{U}$, and write $\beta$ for the infimum of $||\bar{A}^\vee \bar{f}||_1 / ||\bar{f}||_1$ as $\bar{f}$ ranges over $\bar{U}_0$.  Since $||\bar{A}^\vee \bar{f}||_1 / ||\bar{f}||_1$ is invariant under scaling, we can compute the infimum over the compact subset of $\bar{U}_0$ where $||\bar{f}||_1 = 1$.  Since $||\bar{A}^\vee \bar{f}||_1 / ||\bar{f}||_1$ is nonzero on this compact subset, its infimum $\beta$ is greater than $0$.

Let $B$ be a parallelipiped in $\R^{n-1}$ which is a fundamental domain for $P \Lambda$ and let $B_\Z$ be $\Lambda \cap B$, a set of size $P^{n-1}$; then the number of full copies of $B_\Z$ contained in $E_D$ can be expressed as $c D^{n-1} + O(D^{n-2})$, where $c$ is the volume of $\Omega_\Q(1)$ divided by that of $B$.  The remainder of $E_D$ is made up of $O(D^{n-2})$ partial copies.  The periodic function $\tilde{f}$ is the pullback to $V_D$ of some function $\bar{f} \in U$, and the decomposition of $E_D$ into fundamental domains shows that $\|A^\vee \tilde{f}\|_1$ acquires a mass of $(c D^{n-1} + O(D^{n-2})) \|\bar{A}^\vee \bar{f}\|_1$ from the full copies of $B_\Z$, and at most another $O(D^{n-2} \|\bar{A}^\vee \bar{f}\|_1)$ from the partial copies.  In particular, 
\beq
\|A^\vee \tilde{f}\|_1 = \Omega(D^{n-1} \|\bar{A}^\vee \bar{f}\|_1).
\eeq
Modifying $\tilde{f}$ by an periodic element of $\ker A^\vee$, we can arrange that $\bar{f}$ lies in $\bar{U}_0$.  Having done so, the fact that $\beta$ is bounded away from $0$ can be written as
\beq
\|A^\vee \bar{f}\|_1 = \Omega(\|\bar{f}\|_1)
\eeq
Combining these two, we find that
\beq
\|A^\vee \tilde{f}\|_1 = \Omega (D^{n-1} \|\bar{f}\|_1) = \Omega(D^{n-2} \|\tilde{f}\|_1)
\eeq
or, in other words,
\beq
\|\tilde{f}\|_1 = O(D^{2-n} \|A^\vee \tilde{f}\|_1).
\eeq

The same upper bound immediately follows for 
$\|\tilde{f}\|^{int}_1$, since it is bounded above by $\|\tilde{f}\|_1$.

Note that, since $A^\vee \tilde{f}$ is periodic with bounded period, $\|A^\vee \tilde{f}\|^{int}_1$ is bounded below by a constant multiple of $\|A^\vee \tilde{f}\|_1$.  What's more,
$\|A^\vee (f - \tilde{f})\|^{int}_1$
is a sum of values of $|(f-\tilde{f})(m)|$ for $m$ in the interval $V_{(1-\epsilon)D}$, with each $m$ appearing at most $O(D^{n-2})$ times.  So
\beq
\|A^\vee (f - \tilde{f})\|^{int}_1 = O(D^{n-2} \|f-\tilde{f}\|^{int}) = O(\|A^\vee f\|_1).
\eeq
and we conclude that
\beq
\|A^\vee \tilde{f}\|_1 = O(\|A^\vee \tilde{f}\|_1^{int}) = O(\|A^\vee f\|_1^{int} + \|A^\vee (f-\tilde{f}) \|_1^{int} = O(\|A^\vee f\|_1).
\eeq

Putting all this together, we get
\begin{eqnarray*}
\|f\|_1^{int} \leq   \|f-\tilde{f}\|_1^{int} + \|\tilde{f}\|^{int}_1 \\
= O(D^{2-n}\|A^\vee f\|_1) + O(D^{2-n}\|A^\vee \tilde{f}\|_1) \\
= O(D^{2-n}\|A^\vee f\|_1).
\end{eqnarray*}

which was the statement to be proved.

\end{proof}

Now, at long last, we are ready to use the existence of the approximately balanced weighting.  Suppose there exists a $\Gamma$-positive function $f_0$ in $U_D^\vee$.  Then by Proposition~\ref{pr:smalloninterior}, we may choose a function $f$ in the same coset of $\ker A^\vee$ as $f_0$ such that
\beq
\|f\|^{int}_1 = O(D^{2-n} \|A^\vee f\|_1).
\eeq
and which is still $\Gamma$-positive, since modifying by $\ker A^\vee$ doesn't change the nonnegative function $A^\vee f_0$. 

On the other hand, by \eqref{eq:ellipseatfinequality} we have
\beq
\|A^\vee f\|_1 = O(D^{n-3} \|f\|_1)
\eeq
and thus that
\beq
\|f\|_1^{int} = O(D^{-1} \|f\|_1).
\eeq

In other words, the $L^1$ mass of $f$ is concentrated near the edges of $V_D$.  We now show that this is not possible for a $\Gamma$-positive function.

\begin{prop}
    Let $f \in U_D^\vee$ be a $\Gamma$-positive function.  Then $\|f\|_1 = o(D \|f\|_1^{int}).$
    \label{pr:edgeconcentration}
\end{prop}

\begin{proof}
    Write $B$ (for ``boundary") for the complement of $V_{(1-\epsilon)D}$ in $V_D$.  Recall that $\sum_i f_i = 0$.  It follows that for each $b \in B$, there exists at least one $i$ such that $f_i(b) \leq -(1/2(n-1))|f(b)|$; for each $b$ choose such an $i$ and call it $i_b$.  Now let $\Sigma \subset E_D$ be the set of edges such that, for some $b \in B$, we have $e_{i_b} = b$, and such that, for every $j \neq i_b$, we have $e_j \in V_{(1-\epsilon)D}$.  For each $b$, write $\Sigma_b$ for the subset of $\Sigma$ with $e_{i_b} = b$, and for each pair $a \in V_{(1-\epsilon)D}, j \in 1, \ldots, n$, write $\Sigma_{a,j}$ for the subset of $\Sigma$ with $e_j = a$.

    The point of doing this is that for an edge $e$ in $\Sigma$ to have $A^\vee(e) = \sum_i f_i(e_i) \geq 0$ as required, the values of $f(e_j)$ for $j \neq i_b$ must be fairly large in order to counteract the large negative contribution $f_{i_b}(b)$.  That is, we have
    \begin{equation}
    \sum_{j \neq i} |f_j(e_j)| \geq
    |f_{i_b}(b)|.
    \label{eq:bainequality}
    \end{equation}

    And this will put too much of the mass of $f$ in the interior interval $V_{(1-\epsilon)D}$.

    We fix some threshold $Z$, and split $B$ into two regions; write $B_0$ for those $b$ such that $|\Sigma_b| \leq Z$, and $B_1$ for those $b$ with $|\Sigma_b| > Z$.  The reader may wonder at this point how we even know $\Sigma$ is nonempty. Since $\epsilon$ was chosen to be smaller than the $\delta$ in Lemma~\ref{le:deltaedge}, {\em every} edge with $e_{i_b} = b$ has all other $e_j$ in $V_{(1-\epsilon)D}$. Thus, the size of $\Sigma_b$ is simply the number of edges in $E_D$ with $e_{i_b} = b$. 

    We first consider $B_1$.  Sum \eqref{eq:bainequality} over all $\Sigma$ with $b \in B_1$.  On the right hand side, we have
    \beq
    \sum_{b \in B_1} |\Sigma_b| |f_{i_b}(b)|
    \geq \frac{1}{2(n-1)} \sum_{b \in B_1} |\Sigma_b| |f(b)|
    \eeq
    while on the left-hand side, we have a sum of values $|f_j(a)|$.  The number of times each such value can occur is at most the number of edges in $E_D$ with $e_j = a$, which is bounded above by a constant multiple of $D^{n-2}$.  The sum of the left-hand side is thus $O(D^{n-2} \|f\|_1^{int})$.  Putting these together, we find that
    \beq
    Z \|f|B_1\|_1 \leq \sum_{b \in B_1} |\Sigma_b| |f(b)| = O(D^{n-2} \|f\|_1^{int})
    \eeq
    so
    \beq
    \|f|B_1\|_1 = O(Z^{-1} D^{n-2} \|f\|_1^{int}).
    \eeq

    For $B_0$, on the other hand, we choose just {\em one} element $e_b$ of $\Sigma_b$ for each $b \in B_0$.  (Here is where we finally use Proposition~\ref{pr:sameimage}, which tells us that there {\em is} an element of $\Sigma_b$ for each $b \in B_0$.)  We now sum \eqref{eq:bainequality} over the resulting set of $|B_0|$ edges. The right-hand side is
\beq
\sum_{b \in B_0} |f_{i_b}(b)| \geq \frac{1}{2(n-1)}\|f|B_0\|_1.
\eeq
The left-hand side, on the other hand, is a sum of $(n-1)|B_0|$ terms of the form $|f_j(a)|$, and such a sum is certainly $O(|B_0| \|f\|_1^{int})$.  So we have
\beq
\|f|B_0\|_1 = O(|B_0| \|f\|_1^{int}).
\eeq
The tradeoff between these two bounds hinges on the choice of $Z$.  When $Z$ is larger, $|B_0|$ increases, making the upper bound on $\|f|B_0\|_1$ worse, but the upper bound on $\|f|B_1\|_1$, which has $Z$ in the denominator, improves.  The key point is to show that there are not {\em too} many $b$ for which $\Sigma_b$ is very small.  To see this, we recall the computations in the proof of Proposition~\ref{pr:sameimage}.  We showed there that if $b = D\mu - d$ for some integer $d>0$, the set $\Sigma_b$ of all edges with $e_{i_b} = b$ can be expressed as the intersection of a fixed lattice of rank $n-2$ with a translate of a dilate of a fixed ellipsoid by the factor
\beq
(D^2 - (D\mu - d)^2 / \mu^2)^{1/2} =
(2dD\mu^{-1} - d^2 \mu^{-2})^{1/2}
\eeq
which factor, if $d = O(D)$, is $\Omega(d^{1/2} D^{1/2})$.  This implies that
\beq
|\Sigma_b| = \Omega(d^{(n-2)/2} D^{(n-2)/2}).
\eeq

So if for instance we choose $Z = D^{n-2.5}$, we find that $d^{(n-2)/2} D^{(n-2)/2} > Z$ once $d > D^{\frac{n-3}{n-2}}$; in other words, $|B_0| = O(D^{\frac{n-3}{n-2}})$.  We then have 
\beq
\|f|B_1\|_1 = O(Z^{-1} D^{n-2} \|f\|_1^{int}) = O(D^{0.5} \|f\|_1^{int})
\eeq
while
\beq
\|f|B_0\|_1 = O(D^{\frac{n-3}{n-2}}\|f\|_1^{int}).
\eeq
Since $\|f\|_1 = \|f|B_0\|_1 + \|f|B_1\|_1 + \|f\|_1^{int}$, we conclude that
\beq
\|f\|_1 = o(D \|f\|_1^{int})
\eeq
as claimed.
    \end{proof}

We have arrived at the end.  If there exists a $\Gamma$-positive function in $U_D^\vee$, then by Proposition~\ref{pr:smalloninterior}, there is a $\Gamma$-positive function $f$ with $\|f\|^{int}_1 = O(\|f\|_1)$.  But Proposition~\ref{pr:edgeconcentration} shows there is no such $f$.  Thus there are no $\Gamma$-positive functions in $U_D^\vee$ at all once $D$ is sufficiently large, which means that for $D$ sufficiently large, the edges $E_D$ 
of $\Gamma$ admit a balanced weighting, which provides the desired non-deterministic solution of the equation $L_1 \eqd \ldots \eqd L_n$.

\bibliographystyle{alpha}
\bibliography{bib.bib}
\end{document}